\documentclass[twoside,a4paper,10pt]{article}
\usepackage[english]{babel}
\usepackage{cite}
\usepackage{amssymb}
\usepackage{amsmath}
\usepackage{amsthm}
\usepackage{mathtools}
\usepackage{mathrsfs}
\usepackage[mathscr]{euscript}

\usepackage[usenames,dvipsnames]{color}
\definecolor{rosso}{rgb}{1,0,0}

\def\pier #1{#1}

\def\pcol #1{#1}

\iftrue
\usepackage{hyperref}
\else
\newcommand{\texorpdfstring}[2]{#1}
\fi

\newcommand{\myboxed}[1]{#1}
\newcommand{\important}[1]{#1}
\newcommand{\marginnote}[1]{}

\newcommand{\AMSsubject}[2]{#1}

\numberwithin{equation}{section}

\theoremstyle{definition}
\newtheorem{defn}{Definition}
\numberwithin{defn}{section}
\newtheorem*{defn*}{Definition}

\newtheorem{prop}[defn]{Proposition}
\newtheorem{lem}[defn]{Lemma}
\newtheorem{cor}[defn]{Corollary}
\newtheorem{thrm}[defn]{Theorem}
\theoremstyle{remark}
\newtheorem{rmrk}[defn]{Remark}
\newtheoremstyle{examples}%
{}{}{}%
{}{\bfseries}%
{}{\newline}{}
\theoremstyle{examples}
\newtheorem*{exmps*}{Examples}

\newcommand{\de}{\mathrm{d}}

\newcommand{\natr}{\mathbb{N}}

\newcommand{\real}{\mathbb{R}}

\newcommand*\Laplace{\mathop{}%
  \!%
  \mathbin\bigtriangleup}
\newcommand{\norm}[1]{\lVert#1\rVert}

\DeclareMathOperator\Span{span}
\DeclareMathOperator\Sign{Sign}
\DeclareMathOperator\divergence{div}

\newcommand{\quadd}{\quad\quad}

\newcommand{\problem}[1]{{\em Problem}~(#1)}

\title{}
\author{}
\date{}
\begin{document}
\maketitle
\begin{center}
  \Large \bf Sliding mode control for a generalization of the Caginalp
  phase-field system
\end{center}
\begin{center}
  \large \sc Pierluigi Colli and Davide Manini 
\end{center}
\begin{center}
  \large \em Dipartimento di Matematica ``F. Casorati'', Universit\`a di Pavia
  
  Via Ferrata, 1, 27100 Pavia, Italy

  \vskip2mm
  \normalsize
  Email: {\tt pierluigi.colli@unipv.it\quad
    davide.manini01@universitadipavia.it}
\end{center}



\begin{abstract}
\noindent
In the present paper, we present and solve the sliding mode control
(SMC) problem for a second-order generalization of the Caginalp
phase-field system.
\pier{This} generalization, \pier{inspired by the theories developed by Green and Naghdi
on one side, and Podio-Gui\-du\-gli on the other,}
deals with the concept of thermal displacement, i.e.\pier{,} a primitive with
respect to the time of the temperature.
Two control laws are considered: the former forces the solution to
reach a sliding manifold described by a linear constraint between the
temperature and the phase variable; the latter forces the phase
variable to reach a prescribed distribution $\varphi^*$.
We prove existence, uniqueness as well as continuous dependence of the
solutions for both problems; two regularity results are also given.
We also prove that, under suitable conditions, the solutions reach the
sliding manifold within finite time.

\vskip3mm
\noindent
{\bf Key words.}\quad
phase field system,
nonlinear boundary value problems,
phase transition,
sliding mode control,
state-feedback control law.

\vskip3mm
\noindent
{\bf AMS subject classification.}\quad
\AMSsubject{35K55}{Nonlinear parabolic equations},
\AMSsubject{35B30}{Dependence of solutions on initial and boundary data, parameters},
\AMSsubject{80A22}{Stefan problems, phase changes, etc.},
\AMSsubject{34H05}{Control problems},
\AMSsubject{93B52}{Feedback control}.
\end{abstract}

\markboth{}{}

\section{Introduction}
For many years, sliding mode control (SMC) has been recognized as one
of the best approaches for the design of robust controllers
for nonlinear dynamical systems.
Nowadays, SMC is considered a standard tool for the regulation of
time-evolving systems in finite dimension
\cite{bfpu08,eff06,fji11,utkin92,young-ozguner93}.\marginnote{migliorare biblio?}

The design of feedback control
systems with sliding modes involves the construction of suitable control
functions enforcing motion along a given manifold of lower dimension,
called {\em sliding manifold}.
The main
idea is (i) to identify this manifold where the control target is
fulfilled and such that the original system restricted to this sliding
manifold has a desired behavior; (ii)  to act on the system through
a suitable control term in order to constrain the evolution on it.
This new term forces the trajectories of the system to reach
the sliding manifold and maintains them along it.

Sliding mode controls feature robustness with respect to unmodelled
dynamics and insensitivity to external disturbances.
At the same time they are relatively easy to design.
For these reasons, in the last years there has been a growing interest
in bringing these methods for finite-dimensional systems described by
ODEs \cite{levaggi02, orlov00, orlov-utkin83} to the realm of PDEs.
%
While certain early works going in this direction
\cite{orlov-utkin83,orlov-utkin87,orlov-utkin98} deal with particular
classes of PDEs, the theoretical development in a general Hilbert
space setting has gained attention only in the last years
\cite{crs11,levaggi13,xlgk13}. \pcol{We also point out the article~\cite{bcgmr}, 
mainly related to the contents of this paper, and the subsequent contributions \cite{colturato, cgmr, cem}, also dealing with systems of PDEs.} 
%

%
In this paper, the considered system describes the spatial and time fluctuations
close to a phase transition.
In order to take into account the effects of phase dissipation, Caginalp
introduced~\cite{caginalp86} a phase-field system consisting of the
following equations
\begin{alignat}{1}
  &(\vartheta+l\varphi)_t-\kappa\Laplace\vartheta
  =f,\label{eq:heat-caginalp}
  \quadd\text{ in }\Omega\times(0,T),
  \\
  &\varphi_t-\Laplace\varphi +F'(\varphi)=\gamma\vartheta,
  \quadd\text{ in }\Omega\times(0,T).
  \label{eq:phase-caginalp}
\end{alignat}
\pier{Here,} $\Omega\subset\real^N$ represents the spatial domain where the
evolution takes place and $T>0$ is the final time of the evolution.
While the case $N=3$ is the one of physical interest, we will carry
out our analysis for all $N$.
A usual choice, at least for the phase variable $\varphi$, is to
complement the equations with standard homogeneous Neumann boundary
conditions $\partial_n\varphi=\partial_n\vartheta=0$ on
$\partial\Omega\pier{{}\times(0,T)}$, plus the initial conditions
$\vartheta(\cdot,0)=\vartheta_0$ and $\varphi(\cdot,0)=\varphi_0$.

The variable $\vartheta$ represents the relative temperature,
i.e.\ the difference between the actual temperature and the fixed
critical temperature for the phase transition.
The variable $\varphi$ has the meaning of a phase parameter:
$\varphi<0$ indicates one of the two phases; $\varphi>0$ indicates the
other phase\pier{;}
$\varphi(x,t)=0$ usually indicates that position $x$ is at the interface
between the two phases at time $t$.
$F'$ is the
derivative of a double-well potential $F$.
A few examples for the double-well potential $F$
are\marginnote{aggiungerne altre}
\begin{align}
  F(r)&=\frac{1}{4}(r^2-1)^2,
  \\
  \label{eq:logar}
  F(r)&=
  \begin{cases}
    (1+r)\log(1+r)+(1-r)\log(1-r)-(c_0+1)r^2,&\text{ if }|r|<1,\\
    \pier{2\log 2 - c_0 - 1},&\pier{\text{ if }|r|=1,}\\    
    +\infty&\text{ otherwise,}
  \end{cases}
  \\
  \label{eq:potential-indicator}
  F(r)&=
  \begin{cases}
    -c_0r^2,&\text{ if }|r|\leq1,\\
    +\infty,&\text{ otherwise,}
  \end{cases}  
\end{align}
where $c_0\in\real$, $c_0>0$. \pier{Please note that the potential in 
\eqref{eq:logar} has a derivative in $(-1,1)$ becoming singular as $r$ approaches $-1$ or $1$, 
while the potential \eqref{eq:potential-indicator} is non-smooth so that part of its derivative should be replaced by the subdifferential of the indicator function of $[-1,1]$.}

%

The physical equations originating the
system~\eqref{eq:heat-caginalp}--\eqref{eq:phase-caginalp}
are\footnote{Please note that in this paper we will use both the
  notations $p_t$ and $\partial_tp$ to denote the derivative of a
  function $p$.}
\begin{align}
  &\partial_t e+\divergence\mathbf q = \tilde f,
  \label{eq:balance}
  \\&
  \partial_t \varphi+\frac{\delta \mathscr F}{\delta \varphi}=0,
  \label{eq:phase-variational}
\end{align}
where $e$ denotes the internal energy, $\mathbf q$ the thermal flux,
and $\tilde f$ the heat source.
The term $\frac{\delta \mathscr F}{\delta \varphi}$ represents the
variational derivative with respect to $\varphi$ of the following
functional
\begin{equation}
  \mathscr F(\vartheta,\varphi)=
  \int_\Omega\left(-\frac{c_0}{2}\vartheta^2-\gamma\vartheta\varphi
  +F(\varphi)+\frac{1}{2}|\nabla\varphi|^2\right),
\end{equation}
where the constants $c_0$ and $\gamma$ represent the specific heat and
the latent heat coefficient, respectively.
Note that the term $-\gamma\vartheta\varphi$ favors states having
concordant relative temperature and phase variable.
The internal energy $e$ \pier{may} be derived from the functional $\mathscr
F$, taking minus its variational derivative with respect to
$\vartheta$, i.e.,
\begin{equation}
  e=-\frac{\delta\mathscr F}{\delta \vartheta}
  =c_0\vartheta+\gamma\varphi.
\end{equation}
Equation~\eqref{eq:phase-variational} yields
equation~\eqref{eq:phase-caginalp} by standard variational derivative
taking into account the homogeneous Neumann condition for $\varphi$.
We set $l:=\gamma/c_0$ and $f:=\tilde f/c_0$.
If we assume the classic Fourier law
\begin{equation}
  \label{eq:fourier}
\mathbf q=-c_0\kappa\nabla\vartheta,
\end{equation}
equation~\eqref{eq:balance}
yields~\eqref{eq:heat-caginalp}.
The homogeneous Neumann condition for $\vartheta$ follows from the
no-flux condition $\mathbf q\cdot \mathbf n=0$ on the boundary of
$\Omega$.

A sliding-mode analysis has been carried out recently for the system described
by the equations~\eqref{eq:heat-caginalp}--\eqref{eq:phase-caginalp}
\cite{bcgmr}.
In the quoted paper three cases are taken into consideration (labeled
as Problem A--C).
In Problem A, the sliding manifold is given by a linear constraint
between $\vartheta$ and $\varphi$; in Problems B and C the phase
$\varphi$ is forced to reach a prescribed phase distribution
$\varphi^*$.
While in Problems A and B the control law is non-local in the spatial
variable, in Problem C the control term is fully local. 

In the present paper we carry out a similar sliding-mode analysis,
for modified equations where the Fourier law~\eqref{eq:fourier} is
generalized in the light of the works by Green and Naghdi
\cite{green-naghdi91,green-naghdi92,green-naghdi93} and (more
recently) by Podio-Guidugli~\cite{podio-guidugli09}\pier{, referring to Thermodynamics}.
%
These papers introduced the notion of {\em thermal displacement},
which is a primitive of the temperature, i.e.
\begin{equation}
  w(x,t)=w_0(x)+\int_0^t\vartheta(x,s)\de s,
\end{equation}
where $w_0$ represents a given datum accounting for a possible previous
thermal history of the phenomenon.
\marginnote{spigazione di $w_0$}%
Making use of this new variable, 
these authors proposed three theories for heat transmission labeled as
type I--III.
Type I theory, after suitable linearization, yields the standard Fourier law
\begin{equation}
  \mathbf q=-c_0\kappa\nabla w_t\quadd\text{ (type I),}
\end{equation}
which has been studied extensively.
Linearized versions of type II and III give the following
heat-conduction laws\marginnote{sistemare quest'eq.}
\begin{align}
  &\mathbf q=-c_0\tau\nabla w&&\text{ (type II),}
  \\&
  \mathbf q=-c_0\kappa\nabla w_t-c_0\tau\nabla w&&\text{ (type III).}
  \label{eq:type-III}
\end{align}
It is important to note that the thermal displacement $w$ becomes
necessary to describe type II and III laws, whereas type I law can be
described just in terms of the temperature $\vartheta=\partial_t w$.
The role of the primitive $w$ in type II and III theories is to
account for the past thermal history of the heat-conducting body.

This paper focuses on the most general type III theory.
In type III theory, the special $\tau=0$ case reduces to standard
type I theory; $\kappa=0$ yields type II theory.
Equation~\eqref{eq:balance}, along with type
III law~\eqref{eq:type-III}, leads to this formulation
\begin{equation}
  (w_{t}+l\varphi)_t-\kappa\Laplace w_t-\tau\Laplace w=f,
  \label{eq:heat-introduction}
  \quadd\text{ in }\Omega\times(0,T).
\end{equation}
Equation~\eqref{eq:phase-caginalp} with the substitution
$\vartheta=\partial_t w$ becomes
\begin{equation}
  \varphi_t-\Laplace\varphi +F'(\varphi)=\gamma w_t,
  \quadd\text{ in }\Omega\times(0,T).
  \label{eq:phase-introduction}
\end{equation}
The no-flux condition $\mathbf q\cdot\mathbf n = 0$ generates the
homogeneous Neumann boundary condition $\partial_n w=0$.
For the
system~\eqref{eq:heat-introduction}--\eqref{eq:phase-introduction}, well-posedness, asymptotic analysis, and convergence of the solutions
as $\tau\to0$ to the solution of the original Caginalp
system~\eqref{eq:heat-caginalp}--\eqref{eq:phase-caginalp} has
been carried out in \cite{canevari-colli12,canevari-colli13}.

In order to enable the SMC in the system above, we add a feedback term
in either equation~\eqref{eq:heat-introduction}
or~\eqref{eq:phase-introduction} \pier{in order to force} the solutions
$(w(t),\varphi(t))$ to reach the sliding manifold.
In \pier{the present paper}, we consider two cases.
\pier{As for} the first one, we adopt the following linear condition connecting
$w$ and $\varphi$
\begin{equation}
  \partial_t w(t)+\alpha\varphi(t)=\eta^*,
  \label{eq:sliding-manifold}
\end{equation}
to describe the sliding manifold, and the feedback control is added to
the left-hand side of equation~\eqref{eq:heat-introduction}.
In~\eqref{eq:sliding-manifold}, $\alpha$ is a real positive
\marginnote{perch\'e posiiva?} constant and $\eta^*$ a prescribed
function independent of time.
\pier{For} the second case, \pier{our aim is that the phase reaches} 
a prescribed phase distribution $\varphi^*$\pier{; in order to achieve that, 
we insert} the feedback term in~\eqref{eq:phase-introduction}.
The linear condition~\eqref{eq:sliding-manifold}\pier{,} as well as
the choice of the sign operator in $L^2(\Omega)$\pier{,} corresponds
to the Problem (A) in~\cite{bcgmr} for the first case studied there,
whereas the second case we investigate here corresponds to the Problem
(B) in~\cite{bcgmr}.
However, with respect to the arguments used in~\cite{bcgmr}, here we
adopt a slightly different approach, based on the simplification of
the auxiliary lemma and \pier{on the observation that} the
Moreau--Yosida regularization of the norm of a Banach space converges
uniformly. \pier{Thus, we improve some technical aspects
of~\cite{bcgmr}. Moreover, let us point out that our
system~\eqref{eq:heat-introduction}--\eqref{eq:phase-introduction} is
more difficult to handle, due to the hyperbolicity of
equation~\eqref{eq:heat-introduction}.}

This \pier{paper} is organized as follows.
The next section \pier{deals with} the common notation and the considered system
of equations;
it also \pier{contains the precise statements of results. Our theorems 
establish well-posedness, regularity  properties for the solutions, 
and in particular the existence of sliding modes for both problems. 
We note that in our argumentation the} results are grouped by problem.
The remaining sections\pier{, from \S~3 to \S~6,} are devoted to the proofs.
%


\section{Common notation and main results}
In this section, we \pier{set the notation and} present the problems
that we will solve, as well as the results concerning well-posedness
of the problems and regarding SMC.
Moreover, a few technical tools are recalled.

First of all, we require for $\Omega\subset \real^N$ to be an open,
bounded, smooth set. $\Gamma$ and $\partial_n$ represent the boundary
of $\Omega$ and the outward normal derivative on $\Gamma$,
respectively.  We set $Q_t=\Omega\times (0,t)$ for $t\in (0,T]$ and
$Q=Q_T$.

In the sequel, we will make use of techniques of convex analysis, so
we split \pier{the potential} \marginnote{dire prima cos'\`e $F$}
$F$ \pier{as the sum} $\widehat\beta+\widehat\pi$, \pier{by} requiring that
\begin{align}
  &\widehat\beta:\real\to[0,+\infty]
  \text{ is convex, proper, l.s.c.\ with }
  \widehat\beta(0)=0,\label{eq:beta-convex}\\
  &\widehat\pi:\real\to\real\text{ is $C^1$ and $\widehat\pi'$ is
    Lipschitz-continuous}.\label{eq:pi-Lipschitz}
\end{align}
We define $\beta$ and $\pi$ as the
subdifferential~\cite[\S~23]{rockafellar70} of $\widehat\beta$ and the
derivative of $\widehat\pi$, respectively.  It turns out that $\beta$
is a maximal monotone graph of $\real^2$, such that $0\in\beta(0)$.
We indicate with $\beta^\circ(r)$ the element of $\beta(r)$ having
minimum modulus. 

We make the following assumptions on the data of the problem
\begin{alignat}{1}
  &\kappa,\tau,\gamma,l\in (0,+\infty),\label{eq:datanum}\\
  &f\in L^2(Q).\label{eq:dataf}
\end{alignat}
 We introduce the following Hilbert spaces
\begin{equation*}
H:=L^2(\Omega),\quadd V:=H^1(\Omega),\quadd W:=\{v\in H^2(\Omega):
\partial_n v=0\}.
  \nonumber
\end{equation*}
On $H$ and $V$ we put the standard Hilbert norm, while we endow $W$
with the norm $\norm{u}_W^2=\norm{u}_H^2+\norm{\Laplace u}_H^2$, which
is equivalent to the norm $\norm{\cdot}_{H^2(\Omega)}$, by the smooth
boundary condition and elliptic regularity (see
e.g.~\cite[\S~9.6]{brezis11} or~\cite[\S~6.3]{evans98}).
The norm $\norm{\cdot}_H$ will also denote the norm of the space
$H^N=L^2(\Omega;\real^N)$.
The scalar product of $H$ and $H^N$ will be denoted with
$(\cdot,\cdot)$.
We define the \pier{$\Sign$} operator for the Hilbert space $H$ as the
subdifferential of the norm $\norm{\cdot}_H$,
namely\marginnote{biblio}
\begin{equation*}
\Sign(v)=
\begin{cases}
  \cfrac{v}{\norm{v}_{H}} &\text{ if } v\neq 0,\\
  B_H &\text{ if } v=0,
  \end{cases}
\end{equation*}
where $B_H$ is the closed unit ball of $H$.
The regularity hypotheses for the initial data are
\begin{gather}
  \label{eq:initial-data-hypothesis}
  \vartheta_0\in V,\quadd w_0\in W,\quadd\varphi_0\in V,\quadd
  \widehat\beta(\varphi_0)\in L^1(\Omega).
\end{gather}
A solution is a quadruplet $(w,\varphi,\xi,\sigma)$, for which we
require, at least, the following regularity 
\begin{alignat}{1}
  w&\in H^2(0,T;H)\cap W^{1,\infty}(0,T;V)\cap H^1(0,T;W),\label{eq:regw}\\
  \varphi&\in H^1(0,T;H)\cap L^\infty(0,T;V)\cap L^2(0,T;W),\\
  \xi&\in L^2(0,T;H),\\
  \sigma&\in L^\infty(0,T;H).\label{eq:regsigma}
\end{alignat}
We note that the Neumann boundary condition
$\partial_nw=\partial_n\varphi=0$ is incorporated in the definition of
the space $W$.

Given \important{a {\em sliding mode parameter}} $\rho>0$,
$\alpha\in\real$ and a {\em target function} \pier{$\eta^*$} the first problem
is to find a quadruplet $(w,\varphi,\xi,\sigma)$ satifying the
regularity conditions~\eqref{eq:regw}--\eqref{eq:regsigma} and
\begin{alignat}{2}
  &(w_t+l\varphi)_t-\kappa\Laplace w_t -\tau\Laplace w +\rho\sigma =f
  \quadd\text{ a.e.\ in }Q,
  \label{eq:heat}
  \\&
  \sigma\in\Sign(w_t+\alpha\varphi-\eta^*)
  \quadd\text{ in $H$, a.e.\ in }(0,T),
  \label{eq:sign}
  \\&
  \varphi_t-\Laplace\varphi+\xi+\pi(\varphi)=\gamma w_t
  \quadd\text{ a.e.\ in }Q,
  \label{eq:phase}
  \\&
  \xi\in\beta(\varphi)
  \quadd\text{ a.e.\ in }Q,
  \label{eq:beta}
  \\&
  \label{eq:initial-conditions}
  w_t(0)=\vartheta_0,\quadd w(0)=w_0,\quadd \varphi(0)=\varphi_0
  \quadd
  \text{ a.e.\ in }\Omega.
\end{alignat}
We label the problem above as \problem{A}.  The regularity hypothesis
for the target function $\eta^*$ is
\begin{equation}
  \eta^*\in W.
  \label{eq:target-function}
\end{equation}
\important{Let us emphasize that the feedback law is highly non-local,
  as the value of the feedback term at $(x, t)$ depends on $(w(\cdot,
  t), \varphi(\cdot, t))$ and not only on $(w(x, t), \varphi(x, t))$.
  The sliding-mode parameter $\rho > 0$ represents the strength of the
  control law and it plays a central role in this kind of analysis.
  Accordingly, we will highlight the dependence on $\rho$ in all our
  estimates.}
  
On the other hand, given $\rho>0$ and a target function $\varphi^*$,
\problem{B} consists in finding a a quadruplet
$(w,\varphi,\xi,\sigma)$ satifying the regularity
conditions~\eqref{eq:regw}--\eqref{eq:regsigma} and
\begin{alignat}{2}
  &(w_t+l\varphi)_t-\kappa\Laplace w_t-\tau\Laplace w=f
  \quadd\text{ a.e.\ in }Q,
  \label{eq:heatB}
  \\&
  \varphi_t-\Laplace\varphi+\xi+\pi(\varphi)+\rho\sigma=\gamma w_t
  \quadd\text{ a.e.\ in }Q,
  \label{eq:phaseB}
  \\&
  \sigma\in\Sign(\varphi-\varphi^*)
  \quadd\text{ in $H$, a.e.\ in }(0,T),
  \label{eq:signB}
  \\&
  \xi\in\beta(\varphi)
  \quadd\text{ a.e.\ in }Q,
  \label{eq:betaB}
  \\&
  \label{eq:initial-conditionsB}
  w_t(0)=\vartheta_0,\quadd w(0)=w_0,\quadd \varphi(0)=\varphi_0
  \quadd
  \text{ a.e.\ in }\Omega.
\end{alignat}
For this problem, the regularity required for $\varphi^*$ is
\begin{equation}
  \label{eq:target-functionB}
  \varphi^*\in W
  \quadd\text{ and }\quadd
  \pier{\beta^\circ} (\varphi^*)\in H.
\end{equation}
\paragraph{Results for \problem{A}.}

First of all we present \pier{the} well-posedness results, starting from
the existence theorem.
\begin{thrm}[Existence]\label{thrm:existence}
Assume~\eqref{eq:beta-convex}--
\pier{\eqref{eq:initial-data-hypothesis}}, and~\eqref{eq:target-function}.
\marginnote{elencone di ipotesi}
Then
there exist two constants $C_1$, $C_2>0$ such that for every $\rho>0$
the problem~\eqref{eq:heat}--\eqref{eq:initial-conditions} has at least
a solution $(w,\varphi,\xi,\sigma)$
satisfying~\eqref{eq:regw}--\eqref{eq:regsigma} and the 
\pier{estimates} 
\marginnote{attenzione ho aggiunto la stima in $W^{1,\infty}(0,T;V)$}
\begin{align}
  \label{eq:existence-bound1}
  &
  \begin{aligned}
     &\norm{w}_{ W^{1,\infty}(0,T;V)\cap H^1(0,T;W)}
     +\norm{\varphi}_{H^1(0,T;H)\cap L^\infty(0,T;V)\cap L^2(0,T;W)}
     \\%
     &+
     \rho\norm{w_t +\alpha\varphi-\eta^*}_{L^1(0,T;H)}
     +\norm{\widehat\beta(\varphi)}_{L^\infty(0,T;L^1(\Omega))}
     \\%
     &+\norm{\xi}_{L^2(0,T;H)}
     +\norm{\sigma}_{L^\infty(0,T;H)}
     \leq C_1,
  \end{aligned}
  \\[2mm]%
  \label{eq:existence-bound2}
  &\norm{w}_{H^2(0,T;H)}  \leq C_2(1+\rho^{1/2}).
\end{align}
\end{thrm}
The following result gives further regularity of the solutions under
the hypothesis
\begin{equation}\label{eq:further-initial-data}
\varphi_0\in W\quadd\text{ and }\quadd
\beta^\circ(\varphi_0)\in H.
\end{equation}
\begin{thrm}[Further regularity]\label{thrm:further-regularity}
Assume the hypotheses of Theorem~\ref{thrm:existence}
and the condition~\eqref{eq:further-initial-data}. 
Then every solution $(w,\varphi,\xi,\sigma)$ given by
Theorem~\ref{thrm:existence} satisfies
\begin{gather}
\varphi\in
W^{1,\infty}(0,T;H)\cap H^1(0,T;V)\cap L^\infty (0,T;W),
\\
\xi\in L^\infty(0,T;H),
\end{gather}
and 
\begin{equation}\label{eq:further-regularity}
\norm{\varphi}_{
  W^{1,\infty}(0,T;H)\cap H^1(0,T;V)\cap L^\infty (0,T;W)}
+\norm{\xi}_{
  L^\infty(0,T;H)}
\leq C_3(1+\rho^{1/2}),
\end{equation}
where $C_3>0$ is a constant independent of $\rho$.
\marginnote{dire meglio l'indipendenza da $\rho$.}
\end{thrm}

\pier{The regularity given by the above theorem is necessary for
  proving the existence of sliding modes.}

The next \pier{result} closes the topic of the well-posedness of the
problem.  As consequence, we have that under the assumption $l=\alpha$
the solution is unique.
\begin{thrm}[Continuous dependence]\label{thrm:continuous-dependence}
Suppose~\eqref{eq:beta-convex}--\eqref{eq:datanum},
\eqref{eq:target-function}, $\rho>0$, and $l=\alpha$.
Let $i=1, 2$.
We consider
$(\vartheta_{0,i},w_{0,i},\varphi_{0,i},f_i,w_i,\varphi_i,\xi_i,\sigma_i)$,
where the functions $(\vartheta_{0,i},w_{0,i},\varphi_{0,i})$ are
initial data satifying equation~\eqref{eq:initial-data-hypothesis},
$f_i$ is a function satisfying~\eqref{eq:dataf}, and
$(w_i,\varphi_i,\xi_i,\sigma_i)$ is a solution for \problem{B} given by
Theorem~\ref{thrm:existence} with
$(\vartheta_{0},w_{0},\varphi_{0})=(\vartheta_{0,i},w_{0,i},\varphi_{0,i})$
and $f=f_i$.
Then, there exists a constant $C_4>0$, independent of $\vartheta_{0,i},
w_{0,i}, \varphi_{0,i}, f_i$, and $\rho$, such that\marginnote{notare
  l'indipendenza da $\rho$}
\begin{equation}\label{eq:continuous-dependence}
  \begin{split}
    &\norm{w_1-w_2}_{W^{1,\infty}(0,T;H)\cap H^1(0,T;V)}
    +\norm{\varphi_1-\varphi_2}_{L^\infty(0,T;H)\cap L^2(0,T;V)}
    \\&
    \leq C_4(\norm{\vartheta_{0,1}-\vartheta_{0,1}}_H
    +\norm{w_{0,1}-w_{0,2}}_V
    \\&
    +\norm{\varphi_{0,1}-\varphi_{0,2}}_H+\norm{f_1-f_2}_{L^2(Q)}).
  \end{split}
\end{equation}
\end{thrm}
\begin{cor}[Uniqueness]\label{cor:uniqueness}
  Suppose that the hypotheses of Theorem~\ref{thrm:existence} hold
  true and that $l=\alpha$.  Then the solution is unique.
\end{cor}

Finally we come to the most important result for \problem{A}: the theorem
that guarantees that the solutions reach the sliding manifold in a
finite time.
\begin{thrm}[Sliding mode]\label{thrm:sliding-mode}
Assume~\eqref{eq:beta-convex}--\eqref{eq:dataf},
\eqref{eq:initial-data-hypothesis}, \eqref{eq:target-function}
\eqref{eq:further-initial-data}, and $f \in L^\infty(0,T;H)$.  Then
there exist $\rho^*>0$, such that the following condition is
fulfilled\pier{: for} every $\rho>\rho^*$ and for every solution
$(w,\varphi,\xi,\sigma)$ to the
problem~\eqref{eq:heat}--\eqref{eq:initial-conditions} there exist a
time $T^*\in[0,T)$, such that
\begin{equation}
w_t(t)+\alpha\varphi(t)=\eta^*,\quadd\text{ a.e.\ in }\Omega,\text{ for
  a.a.\ }t\in(T^*,T).    
\end{equation}
\end{thrm}

\begin{rmrk}
  \label{rmrk:sliding-mode}
  \marginnote{fare un remark simile per il caso B. o meglio, unificare
    in un unico remark.}
  The statement of the \pier{above} theorem gives no estimates for $\rho^*$ and
  $T^*$, but in the proof we will find certain bounds which we summarize
  here.  Define $C_5$ and $\psi_0$ as
  \begin{align}
    \label{eq:c5-definition}
    &C_5=\tau C_1+(\kappa\alpha+|\alpha-l|)C_3+\kappa\norm{\Laplace\eta^*}_H
    +\norm{f}_{L^\infty(0,T;H)},
    \\&
    \label{eq:psi0-def}
    \psi_0=\norm{\vartheta_0+\alpha\varphi_0-\eta^*}_H,
  \end{align}
  where the constants $C_1$ and $C_3$ are given by
  Theorems~\ref{thrm:existence} and~\ref{thrm:further-regularity},
  respectively.  The quantity $\psi_0$ measures how far the initial
  state is from the sliding manifold.  We will see that it is sufficient
  to choose
  \begin{equation}
    \label{eq:bound-on-rho*}
    \rho^*=2\left(\frac{\psi_0}{T}
    +C_5+\frac{C_5^2}{2}\right)
  \end{equation}
  to fulfill the condition described by the theorem.  Moreover, for a
  given $\rho>\rho^*$, we will prove the following bound on $T^*$\pier{:}
  \begin{equation}
    \label{eq:bound-on-T*}
    T^*\leq\frac{2\psi_0}
    {\rho-2C_5-C_5^2}<T.
  \end{equation}

\end{rmrk}

\paragraph{Results for \problem{B}.}
For \problem{B} we have two \pier{theorems regarding well-posedness}.
\begin{thrm}[Existence]
\label{thrm:existenceB}
Assume~\eqref{eq:beta-convex}--
\pier{\eqref{eq:initial-data-hypothesis}} and~\eqref{eq:target-functionB}.
Then there exist two constants $C_6,\, C_7>0$, such that for every
$\rho>0$ the problem~\eqref{eq:heatB}--\eqref{eq:initial-conditionsB}
has at least a solution $(w,\varphi,\xi,\sigma)$
satisfying~\eqref{eq:regw}--\eqref{eq:regsigma} and the 
\pier{estimates} 
  \begin{align}
    &
    \begin{aligned}
    &
    \norm{w}_{W^{1,\infty}(0,T;H)\cap H^1(0,T;V)}
    +\norm{\varphi}_{L^\infty(0,T;H)\cap L^2(0,T;V)}
    +\norm{\sigma}_{L^\infty(0,T;H)}
    \\&
    +\rho\norm{\varphi-\varphi^*}_{L^1(0,T;H)}
    \leq C_6,
    \end{aligned}
    \label{eq:existence-bound1B}
  \\[2mm]&
  \begin{aligned}
    &
    \norm{w}_{H^2(0,T;H)\cap W^{1,\infty}(0,T;V)\cap H^1(0,T;W)}
    +\norm{\varphi}_{
      H^1(0,T;H)\cap L^\infty(0,T;V)\cap L^2(0,T;W)}
    \\&
    +\norm{\xi
      +\rho\sigma}_{L^2(0,T;H)}
    +\norm{\widehat\beta(\varphi)}_{L^\infty(0,T;L^1)}
    \\&
    +\rho\norm{\varphi(t)-\varphi^*}_{L^\infty(0,T;H)}
    \leq C_7(1+\rho^{1/2}).
    \label{eq:existence-bound2B}
  \end{aligned}
  \end{align}
  \pier{Moreover,} the components $w$ and $\varphi$ of the solution are
  uniquely identified.
  Furthermore, the components $\xi$ and $\sigma$ are uniquely
  identified as well, provided that $\beta$ is single-valued.
\end{thrm}

\begin{thrm}[Continuous dependence]\label{thrm:continuous-dependenceB}
Suppose~\eqref{eq:beta-convex}--\eqref{eq:datanum},
\eqref{eq:target-functionB}, $\rho>0$, and $l=\alpha$.
Let $i=1, 2$.
We consider
$(\vartheta_{0,i},w_{0,i},\varphi_{0,i},f_i,w_i,\varphi_i,\xi_i,\sigma_i)$,
where the functions $(\vartheta_{0,i},w_{0,i},\varphi_{0,i})$ are
initial data satifying equation~\eqref{eq:initial-data-hypothesis},
$f_i$ is a function satisfying~\eqref{eq:dataf}, and
$(w_i,\varphi_i,\xi_i,\sigma_i)$ is a solution for \problem{B} given by
Theorem~\ref{thrm:existenceB} with
$(\vartheta_{0},w_{0},\varphi_{0})=(\vartheta_{0,i},w_{0,i},\varphi_{0,i})$
and $f=f_i$.
Then, there exists a constant $C_8>0$, independent of $\vartheta_{0,i},
w_{0,i}, \varphi_{0,i}, f_i$, and $\rho$, such that\marginnote{notare
  l'indipendenza da $\rho$}
\begin{equation}\label{eq:continuous-dependenceB}
  \begin{split}
    &\norm{w_1-w_2}_{W^{1,\infty}(0,T;H)\cap H^1(0,T;V)}
    +\norm{\varphi_1-\varphi_2}_{L^\infty(0,T;H)\cap L^2(0,T;V)}
    \\&
    \leq C_8(\norm{\vartheta_{0,1}-\vartheta_{0,1}}_H
    +\norm{w_{0,1}-w_{0,2}}_V
    \pier{{}+\norm{\varphi_{0,1}-\varphi_{0,2}}_H+\norm{f_1-f_2}_{L^2(Q)}).}
  \end{split}
\end{equation}
\end{thrm}

For the sake of completeness, we give a regularity result similar to
the one of \problem{B}, although it is not necessary for proving the
sliding mode \pier{result.}

\begin{thrm}[Further regularity]\label{thrm:further-regularityB}
Assume the \pier{same hypotheses as in} Theorem~\ref{thrm:existenceB} and the
condition~\eqref{eq:further-initial-data}.
Then the components $\varphi$ and $\xi$ of a solution $(w,\varphi,\xi,\sigma)$
given by Theorem~\ref{thrm:existenceB} satisfy
\begin{gather}
\varphi\in
W^{1,\infty}(0,T;H)\cap H^1(0,T;V)\cap L^\infty (0,T;W),
\\
\xi\in L^\infty(0,T;H)
\end{gather}
and 
\begin{equation}\label{eq:further-regularityB}
\norm{\varphi}_{
  W^{1,\infty}(0,T;H)\cap H^1(0,T;V)\cap L^\infty (0,T;W)}
+\norm{\xi}_{
  L^\infty(0,T;H)}
\leq C_9(1+\rho),
\end{equation}
where $C_9>0$ is a constant independent of $\rho$.
\marginnote{dire meglio l'indipendenza da $\rho$.}
\end{thrm}

\begin{thrm}[Sliding mode]\label{thrm:sliding-modeB}
Assume~\eqref{eq:beta-convex}--
\pier{\eqref{eq:initial-data-hypothesis}} and~\eqref{eq:target-functionB}.
Then there exist $\rho^*>0$, such that the following condition is
fulfilled\pier{: for} every $\rho>\rho^*$ and for every solution
$(w,\varphi,\xi,\sigma)$ to the
problem~\eqref{eq:heatB}--\eqref{eq:initial-conditionsB} there \pier{exists} a
time $T^*\in[0,T)$, such that
\begin{equation}
  \varphi(t)=\varphi^*,\quadd\text{ a.e.\ in }\Omega,
  \text{ for a.a.\ }t\in(T^*,T).
\end{equation}
\end{thrm}

\begin{rmrk}
We can make \pier{a remark similar to Remark}~\ref{rmrk:sliding-mode}.
In this case we define
\begin{align}
  \label{eq:c9-definition}
  &C_{10}=\gamma C_6+\norm{\beta(\varphi^*)}_H+\mathrm{Lip}(\pi) C_6
  +\pi(0)(T|\Omega|)^{1/2}
  +\norm{\Laplace\varphi^*}_H,
  \\&
  \label{eq:psi0-definitionB}
  \psi_0=\norm{\varphi_0-\varphi^*}_H,
\end{align}
where the constant $C_6$ is given by Theorem~\ref{thrm:existenceB} and
$\mathrm{Lip}(\pi)$ denotes the Lipschitz constant of $\pi$.
We may choose
\begin{equation}
  \rho^*=\frac{\psi_0}{T} + C_{10},
\end{equation}
and we have the following \pier{bound for $T^*$:}
\begin{equation}
  T^*\leq\frac{\psi_0}{\rho-C_{10}}< T.
\end{equation}
\end{rmrk}

We recall for the reader's convenience a slightly modified version of
the Young inequality.
For $a, b, r\in \real$ and $a, b, r>0$ it holds true that
\begin{equation}
  \label{eq:modified-young}
  ab\leq\frac{1}{2r}a^2+\frac{r}{2}b^2.
\end{equation}
We refer to the equation above as the Young inequality and it will be
widely used throughout this work.
For simplicity, we will often omit to write $\de x$, $\de s$, etc., in
integrals.
In order to keep the formulae clear and to avoid boring calculations,
we will use the symbol $C$ to denote a large-enough constant.
This means that $C$ may change from line to line and even in the same
chain of inequalities, and that its value is chosen to satisfy the
inequality where it appears.
While $C$ may depend only on the data of the problem, e.g.\ $\Omega$,
$T$ and the initial data, \pier{$C$ is always} independent of $\rho$ and the
later-introduced parameter $\varepsilon$.
In the proofs of continuous dependence of the solutions, $C$ will be
also independent of the data \pier{corresponding to the single solutions}.

\section{Existence proofs}
The proofs of the existence \pier{results} in Theorems~\ref{thrm:existence}
and~\ref{thrm:existenceB} go along the \pier{following line}.
First of all, we will introduce the Yosida approximations of $\Sign$
and $\beta$.
In the following subsection, we will make use of the Faedo--Galerkin
method, in order to approximate the solutions.
Then we will make certain a priori estimates that give uniform bounds
\pier{on} the approximate solutions.
Finally we will take the limit of the approximate solutions and we
will prove that the limit is actually a solution to the problem.

\subsection{Yosida approximations}
In this subsection we recall a few facts regarding the theory of the
Yosida approximation of maximal monotone operators and the
Moreau--Yosida regularization of convex functions (see
e.g.~\cite[Ch.~2]{brezis73} or~\cite[Ch.~2]{barbu10} for an
introduction to Yosida approximation; see~\cite[Ch.~15]{dalmaso93} for
Yosida regularization in metric spaces).
After considering the abstract case, we will soon apply the results to
the functions $\widehat\beta$ and $\norm{\cdot}_H$.

We start with the definition of Moreau--Yosida regularization.
Given a Hilbert space $X$ (whose norm is denoted by $\norm{\cdot}$), a
proper, convex, l.s.c.\ function $\Phi:X\to[0,+\infty]$, and
$\varepsilon>0$, we define the Moreau--Yosida regularization
$\Phi_\varepsilon$ as
\begin{equation}
  \label{eq:moreau-yosida-def}
  \Phi_\varepsilon(v)=\inf_{w\in X}
  \left\{\tfrac{1}{2\varepsilon}\norm{v-w}^2+\Phi(w)\right\}.
\end{equation}
We incidentally notice  that the infimum in the
definition above is attained.
%
The following proposition \pier{summarizes} the properties, which \pier{will be used}
later on, of the Moreau--Yosida
regularization.
\begin{prop}
  \label{prop:yosida}
  Let $\Phi:X\to[0,+\infty]$ be a convex, proper, l.s.c.\ function.
  Then, the following conclusions hold
  \begin{enumerate}
  \item $\Phi_\varepsilon$ is convex and continuous;
  \item
    $\Phi_\varepsilon(v)<+\infty$ and
    $\Phi_\varepsilon(v)\leq\Phi(x)$ for all $v\in X$;
  \item
    $\Phi_\varepsilon(v)$ converges monotonically to $\Phi(v)$ as
    $\varepsilon\to0$;
  \item
    $\liminf_{\varepsilon\to_0}\Phi_\varepsilon(v_\varepsilon)\geq\Phi(v)$,
    if $v_\varepsilon$ is a sequence converging to $v$;
  \item $\Phi_\varepsilon$ is Fr\'echet-differentiable, the
    differential $\partial\Phi_\varepsilon$ is
    $\varepsilon^{-1}$-Lipschitz-continu\-ous, and
    \begin{equation*}
      \norm{\partial\Phi_\varepsilon(x)}\leq\norm{y}
      \quadd\forall x\in X, y\in\partial\Phi(x).
    \end{equation*}
  \end{enumerate}
\end{prop}
The differential $\partial\Phi_\varepsilon$ coincides with the Yosida
approximation of the maximal monotone operator $\partial\Phi$.

At this point, we introduce $\widehat\beta_\varepsilon:\real\to\real$
and
$\beta_\varepsilon=\partial\widehat\beta_\varepsilon:\real\to\real$,
the Moreau--Yosida regularization of $\widehat\beta$ and Yosida
approximation of $\beta$, respectively.
It follows immediately from the previous proposition and the fact that
$\widehat\beta(0)=0$ and $\beta(0)\ni0$, that
\begin{alignat}{1}
  &
  \beta_\varepsilon(0)=0,\quadd\widehat\beta_\varepsilon(0)=0,
  \\&
  \label{eq:beta-lipschitz}
  |\beta_\varepsilon(r)-\beta_\varepsilon(s)|
  \leq\frac{1}{\varepsilon}|r-s|,\quadd
  0\leq\widehat\beta_\varepsilon(r)\leq\frac{1}{2\varepsilon}r^2,
  \\&
  \label{eq:standard-yosida-for-beta}
  |\beta_\varepsilon(r)|\leq |\beta^\circ(r)|,
  \quadd
  \widehat\beta_\varepsilon(r)\leq\widehat\beta(r),
\end{alignat}
for all $t,s\in\real$, where $\beta^\circ(r)$ denotes the element of
$\beta(r)$ having minimum modulus.

In the same way, we introduce the Moreau--Yosida regularization
$\norm{\,\cdot\,}_{H,\varepsilon}:H\to\real$ and the Yosida approximation
$\Sign_\varepsilon:H\to H$.
It holds that
\begin{equation}\label{eq:quasi-norma-def}
  \norm{v}_{H,\varepsilon}:=
  \min_{w\in H}\{\tfrac{1}{2\varepsilon}\norm{v-w}^2_H+\norm{w}_H\}=
\begin{cases}
  \norm{v}_H -\cfrac{\varepsilon}{2}& \text{if } \norm{v}_H\geq \varepsilon,\\
  \cfrac{\norm{v}_H^2}{2\varepsilon}&\text{if }\norm{v}_H\leq \varepsilon.
\end{cases}
\end{equation}
Indeed, if we differentiate the convex function
$w\mapsto\tfrac{1}{2\varepsilon}\norm{v-w}^2_H+\norm{w}_H$, we obtain
\begin{equation*}
  w+\varepsilon\Sign w\ni v,
\end{equation*}
yielding
\begin{equation*}
  w=
  \begin{cases}
  (1-\tfrac{\varepsilon}{\norm{v}_H})v& \text{if } \norm{v}_H\geq \varepsilon,\\
  0&\text{if }\norm{v}_H\leq \varepsilon,
  \end{cases}
\end{equation*}
thus we can substitute $w$ in the minimum of equation~\eqref{eq:quasi-norma-def}.
We calculate the Yosida approximation of $\Sign$, by
differentiating~\eqref{eq:quasi-norma-def}, obtaining
\begin{equation}
  \label{eq:sign-def}
  \begin{aligned}
    &\Sign_\varepsilon(v)
    =\cfrac{v}{\max\{\varepsilon,\norm{v}_H\}}
    =
    \begin{cases}
      \cfrac{v}{\norm{v}_H}&\text{if }\norm{v}_H\geq \varepsilon,\\
      \cfrac{v}{\varepsilon}&\text{if }\norm{v}_H\leq \varepsilon,
    \end{cases}
  \end{aligned}
\end{equation}
which imply
\begin{equation}\label{eq:quasi-norma}
  (\Sign_\varepsilon(v),v)\geq\norm{v}_{H,\varepsilon}.
\end{equation}
Finally, we point out that the Moreau--Yosida regularization converges
uniformly in $H$, i.e.
\begin{equation}
  \label{eq:quasi-norm-uniform}
  \sup \{\norm{v}_H-\norm{v}_{H,\varepsilon}:{v\in H}\}
  \leq \frac \varepsilon 2.
\end{equation}

\subsection{Existence of solutions for \problem{A}}
\paragraph{Faedo--Galerkin approximation.}
In order to use the Faedo--Galerkin method, we need to introduce a few
notations.  We take $\{v_i\}_{i=1}^{+\infty}$ a complete orthogonal
set of $V$ given by the eigenfunctions of the Laplace operator coupled
with Neumann conditions, i.e.
\begin{equation*}
-\Laplace v_i=\lambda_i v_i
\text{ on }\Omega,
\quadd\partial_n v_i=0
\text{ on }\Gamma,
  \nonumber
\end{equation*}
where $\lambda_i\leq\lambda_{i+1}$, $i\in\natr$, are the eigenvalues
of the Laplace operator.
We define $V_n:=\Span\{v_1, \dots,v_n\}$ and let $P_n:V\to V$ be the
orthogonal projector on $V_n$.
We know that $\cup_{i=1}^{+\infty}V_n$ is dense in $V$.
It is still true that $\{v_i\}_{i=1}^{+\infty}$ is a complete
orthogonal set for $H$ and $W$.
Moreover, the operator $P_n$ can be extended or restricted to $H$ and
$W$ respectively and the extension and the restriction are still
orthogonal projectors in the spaces $H$ and $W$.
We recall that if $v\in X$ then
\begin{equation}\label{eq:proiezioni}
  P_n(v)\to v \text{ strongly in }X
  \quadd\text{ and }\quadd
  \norm{P_n(v)}_X\leq\norm{v}_X,
\end{equation}
where $X$ can be either $H$, $V$ or $W$.
Using standard density results, we take $f_n\in C^0([0,T];H)$,
such that $f_n$ converges strongly to $f$ in $L^2(0,T;H)$.
We now project the initial
data as well as the target function~$\eta^*$\pier{:}
\begin{gather}
\vartheta_{0,n}:=P_n \vartheta_0,\quadd
w_{0,n}:=P_n w_0,\quadd
\varphi_{0,n}:=P_n \varphi_0,
\nonumber
\\[2mm]
\eta^*_{n}:=P_n \eta^*.
\label{eq:proj-target-function}
\end{gather}


The new problem is now to find two functions $w_n\in C^2([0,T];V_n)$
and $\varphi_n\in C^1([0,T];V_n)$, such that\marginnote{formulazione
  variazionale}
\begin{align}
  \label{eq:heat-faedo-galerkin}
  &
  \begin{aligned}
  &(\partial_t^2w_n+l\partial_t\varphi_n
  -\kappa\Laplace \partial_t w_n-\tau\Laplace w_n
  \\
  &+\rho\Sign_\varepsilon(\partial_t w_n+\alpha\varphi_n-\eta_n^*),v)
  =(f_n,v),
  \quadd\forall v\in V_n,\text{ in }[0,T],
  \end{aligned}
  \\[2mm]
  \label{eq:phase-faedo-galerkin}
  &(\partial_t\varphi_n
  -\Laplace \varphi_n
  +\beta_\varepsilon(\varphi_n)
  +\pi(\varphi_n),v)
  =\gamma(\partial_t w_n,v),
  \quadd\forall v\in V_n,\text{ in }[0,T],
  \\[2mm]
  &\partial_t w_n(0)=\vartheta_{0,n}
,\quadd
w_n(0)=w_{n,0}
,\quadd
\varphi_n(0)=\varphi_{n,0}.  
\end{align}
This is a non-linear system of ordinary differential equations of the
second and first order in the variables $w_n$ and $\varphi_n$
respectively.  The non-linearity is only given by $\Sign_\varepsilon$,
$\beta_\varepsilon$, and $\pi$, which are all Lipschitz-continuous
functions.  Hence, by Cauchy--Lipschitz theorem, there exists a unique
solution $(w_n,\varphi_n)$ defined on $[0,T]$.

\paragraph{First a priori estimate.}
We test equation~\eqref{eq:heat-faedo-galerkin} and
equation~\eqref{eq:phase-faedo-galerkin} by taking
$v=\partial_tw_n+\alpha\varphi_n-\eta^*_n$ and
$v=\partial_t\varphi_n$.  We sum with $\frac{1}{2}\frac{\de}{\de
  t}\norm{\varphi_n}_H^2-(\varphi_n,\partial_t\varphi_n)=0$ obtaining
\begin{equation*}
\begin{aligned}
  &\cfrac{1}{2}
  \cfrac{\de}{\de t}
  \norm{\partial_t w_n}_H^2
  +(\partial_t^2w_n,\alpha\varphi_n-\eta^*_n)
  +l(\partial_t\varphi_n,\partial_tw_n+\alpha\varphi_n-\eta^*_n)
  \\
  &+\kappa\norm{\partial_t\nabla w_n}_H^2
  +\kappa(\partial_t \nabla w_n,\nabla(\alpha\varphi_n-\eta^*_n))
  \\
  &+\frac{\tau}{2}
  \cfrac{\de}{\de t}
  \norm{\nabla w_n}_H^2
  +\tau(\nabla w_n,\nabla(\alpha\varphi_n-\eta^*_n))
  \\
  &+\rho(\Sign_\varepsilon(\partial_tw_n+\alpha\varphi_n-\eta^*_n),
  \partial_tw_n+\alpha\varphi_n-\eta^*_n)
  \\&
  +\norm{\partial_t\varphi_n}_H^2
  +\cfrac{1}{2}
  \cfrac{\de}{\de t}
  (\norm{\varphi_n}_H^2+\norm{\nabla\varphi_n}_H^2)
  \\&
  +\cfrac{\de}{\de t}
  \int_\Omega\widehat\beta_\varepsilon(\varphi_n)
  +((\pi(\varphi_n)-\varphi_n),\partial_t\varphi_n)
  \\
  &=(f_n(t),\partial_tw_n+\alpha\varphi_n-\eta^*_n)
  +\gamma(\partial w_n,\partial_t\varphi_n).
\end{aligned}
\end{equation*}
We integrate between $0$ and $t$, and, recalling that
$(\Sign_\varepsilon(v),v)\geq\norm{v}_{H,\varepsilon}$, we get
\begin{equation}\label{eq:big-estimate}
  \begin{aligned}
    &
    \cfrac{1}{2}
    \norm{\partial_t w_n(t)}_H^2
    +\kappa\int_0^t \norm{\partial_t\nabla w_n}_H^2
    +\cfrac{\tau}{2}
    \norm{\nabla w_n(t)}_H^2
    \\&
    +\rho\int_0^t\norm{\partial_tw_n+\alpha\varphi_n+\eta^*_n}_{H,\varepsilon}
    +\int_0^t\norm{\partial_t\varphi_n}_H^2
    \\&
    +\cfrac{1}{2}
    \norm{\varphi_n(t)}_V^2
    +\int_\Omega\widehat\beta_\varepsilon(\varphi_n(t))
    \\&
    \leq
    \frac{1}{2}
    \norm{\vartheta_{0,n}}_H^2
    +\cfrac{\tau}{2}
    \norm{\nabla w_{0,n}}_H^2
    +\cfrac{1}{2}
    \norm{\varphi_{0,n}}^2_V
    +\norm{\widehat\beta_\varepsilon(\varphi_{0,n})}_{L^1(\Omega)}
    \\&
    +\int_0^t (\partial_tw_n,\alpha\partial_t\varphi_n)
    +(\vartheta_{0,n},\alpha\varphi_{0,n} -\eta^*_n)
    -(\partial_t w_n(t),\alpha\varphi(t)-\eta^*_n)
    \\&
    -l\int_0^t(\partial_t\varphi_n,\partial_tw_n+\alpha\varphi_n-\eta^*_n)
    -\kappa\int_0^t (\partial_t\nabla w_n,\nabla(\alpha\varphi_n-\eta^*_n))
    \\&
    -\tau\int_0^t (\nabla w_n,\nabla(\alpha\varphi_n-\eta^*_n))
    -\int_0^t(\pi(\varphi_n)-\varphi_n,\partial_t\varphi_n)
    \\&
    +\int_0^t(f_n,\partial_tw_n+\alpha\varphi_n-\eta^*_n)
    +\int_0^t(\partial_t w_n,\partial_t \varphi_n).
  \end{aligned}
\end{equation}
We need now to control the summands of the left side
of~\eqref{eq:big-estimate}.  By~\eqref{eq:proiezioni} we have that
\begin{equation*}
  \norm{\vartheta_{0,n}}_H^2\leq\norm{\vartheta_0}_H^2\leq C,
  \nonumber
\end{equation*}
and \pier{similarly} we can control $\norm{\nabla w_{0,n}}_H$ and
$\norm{\varphi_{0,n}}_V$.  For the last initial datum we note
\begin{equation}\label{eq:annoying}
  \int_\Omega\widehat\beta_\varepsilon(\varphi_n)
  \leq
  \int_\Omega\cfrac{1}{2\varepsilon}|\varphi_{0,n}|^2
  \leq
  \frac{1}{2\varepsilon}\norm{\varphi_0}_H^2.
\end{equation}
Using the Young \pier{inequality}~\eqref{eq:modified-young} we find
\begin{equation*}
  \int_0^t (\partial_tw_n,\alpha\partial_t\varphi_n)
  \leq
  \cfrac{\alpha^2}{2}\int_0^t\norm{\partial_tw_n}_H^2
  +\cfrac{1}{2}\int_0^t\norm{\partial_t\varphi_n}_H^2.
  \nonumber
\end{equation*}
The next step is easier \pier{as}
\begin{equation*}
  \begin{aligned}
  (\vartheta_{0,n},\alpha\varphi_{0,n} -\eta^*_n)
  &\leq
  \norm{\vartheta_{0}}_H^2
  +\cfrac{\alpha^2}{2}\norm{\varphi_{0}}_H^2
  +\cfrac{1}{2}\norm{\eta^*}_H^2
  \leq C.
  \end{aligned}
  \nonumber
\end{equation*}
Again, owing to \pier{the} Young inequality, we infer
\begin{equation*}
\begin{aligned}
  -(\partial_tw_n(t),\alpha\varphi_n(t)-\eta^*_n)&\leq
  \cfrac{1}{4}
  \norm{\partial_tw_n(t)}_H^2+
  \norm{\alpha\varphi_n(t)-\eta^*_n}_H^2\\
  &\leq
  \cfrac{1}{4}
  \norm{\partial_tw_n(t)}_H^2+
  2\alpha^2
  \norm{\varphi_n(t)}_H^2+C.
\end{aligned}
  \nonumber
\end{equation*}
Since
\begin{equation*}
  \begin{aligned}
    \norm{\varphi_n(t)}_H^2
    &=
    \norm{\varphi_{0,n}}_H^2
    +2 \int_0^t (\varphi_n,\partial_t\varphi_n)
    \\
    &\leq
    \norm{\varphi_0}_H^2
    +8\alpha^2 \int_0^t \norm{\varphi_n}_H^2
    +\cfrac{1}{8\alpha^2}\int_0^t\norm{\partial_t\varphi_n}_H^2,
  \end{aligned}
  \nonumber
\end{equation*}
we find
\begin{equation*}
  -(\partial_tw_n(t),\alpha\varphi_n(t)-\eta^*_n)
  \leq
  \cfrac{1}{4}
  \norm{\partial_tw_n(t)}_H^2
  +\cfrac{1}{4}\int_0^t\norm{\partial_t\varphi}_H^2
  +C\left(1+\int_0^t \norm{\varphi_n}_H^2\right).
  \nonumber
\end{equation*}

Using the same technique we deduce
\begin{gather*}
  \begin{split}
  &-l\int_0^t(\partial_t\varphi_n,\partial_tw_n+\alpha\varphi_n-\eta^*_n)\\
  &\leq
  \cfrac{1}{8}
  \int_0^t
  \norm{\partial_t\varphi_n}_H^2
  +C\left(1+
  \int_0^t
  \norm{\partial_tw_n}_H^2
  +
  \int_0^t
  \norm{\varphi_n}_H^2\right),
  \end{split}\\
  -\kappa\int_0^t (\partial_t\nabla
  w_n,\nabla(\alpha\varphi_n-\eta^*_n)
  \leq
  \cfrac{\kappa}{2}
  \int_0^t\norm{\partial_t\nabla w_n}_H^2
  +\alpha^2\int_0^t\norm{\varphi_n}_H^2
  +C,\\
  -\tau
  \int_0^t (\nabla w_n,\nabla(\alpha\varphi_n+\eta^*_n)
  \leq
  \frac{\tau}{2}\int_0^t\norm{\nabla w_n}_H^2
  +\tau\alpha^2\int_0^t\norm{\nabla\varphi_n}_H^2+C.
\end{gather*}
Then, recalling that $\pi$ is Lipschitz-continuous, we have
\begin{gather*}
  \begin{split}
  -\int_0^t(\pi(\varphi_n)-\varphi_n,\partial_t\varphi_n)
  &\leq
  4\int_0^t\norm{\pi(\varphi_n)-\varphi_n}_H^2
  +\frac{1}{16}
  \int_0^t\norm{\partial_t\varphi_n}_H^2\\
  &\leq C\left(1+\int_0^t\norm{\varphi_n}_H^2\right)
  +\frac{1}{16}
  \int_0^t\norm{\partial_t\varphi_n}_H^2,
  \end{split}\\
  \int_0^t(f_n,\partial_tw_n+\alpha\varphi_n-\eta^*_n)
  \leq C\left(1+
  \int_0^t\norm{\partial_tw_n}_H^2
  +\int_0^t\norm{\varphi_n}_H^2\right),\\
  \int_0^t(\partial_t w_n,\partial_t \varphi_n)
  \leq
  8\int_0^t\norm{\partial_t w_n}_H^2
  +\frac{1}{32}\int_0^t\norm{\partial_t \varphi_n}_H^2.
\end{gather*}
We put everything together, obtaining
\begin{align*}
  &
  \cfrac{1}{32}
  \norm{\partial_t w_n(t)}_H^2
  +\cfrac{\kappa}{2}\int_0^t \norm{\partial_t\nabla w_n}_H^2
  +\cfrac{\tau}{2}
  \norm{\nabla w_n(t)}_H^2
  \\&
  +\rho\int_0^t\norm{\partial_tw_n+\alpha\varphi_n+\eta^*_n}_{H,\varepsilon}
  +\frac{1}{4}\int_0^t\norm{\partial_t\varphi_n}_H^2
  +\cfrac{1}{2}
  \norm{\varphi_n(t)}_V^2
  \\&
  +\int_\Omega\widehat\beta_\varepsilon(\varphi_n(t))
  \leq
  C\left(1+\varepsilon^{-1}
  +\int_0^t\norm{\partial_tw_n}_H^2
  +\int_0^t\norm{\varphi_n}_V^2
  +\int_0^t\norm{\nabla w_n}_H^2.
  \right)
\end{align*}
We use now the Gronwall lemma deducing
\marginnote{dopo i boxed gli levo}
\begin{equation}\label{eq:first-a-priori-estimate}
\myboxed{
  \begin{aligned}
    &
    \norm{w_n}_{
    W^{1,\infty}(0,T;H)\cap
    H^1(0,T;V)
  }
  +\norm{\varphi_n}_{
    H^1(0,T;H)\cap
    L^\infty(0,T;V)}
  \\&
  +\rho\int_0^T\norm{\partial_tw_n+\alpha\varphi_n+\eta^*_n}_{H,\varepsilon}
  +\norm{\widehat\beta_\varepsilon(\varphi_n)}_{
    L^\infty(0,T;L^1(\Omega))}
  \leq C(1+\varepsilon^{-1}).
  \end{aligned}
}
\end{equation}
\begin{rmrk}
\label{rmrk:epsilon-dependence}
Unfortunately, we are not able to provide an estimate independent of
$\varepsilon$, preventing the possibility of taking the limit as
$\varepsilon\to 0$.  For the moment, we do not worry about this
trouble since our first aim is to take the limit as $n\to\infty$.

We stress that the only dependence on $\varepsilon$
in~\eqref{eq:first-a-priori-estimate} arises in~\eqref{eq:annoying}.
We anticipate that the other estimates have a dependence on
$\varepsilon$ because we will use this estimate to prove them.  Hence,
when we will be able to fine-tune equation~\eqref{eq:annoying}
removing the dependence on $\varepsilon$, all estimates will work
perfectly, being independent of $\varepsilon$.

Finally, we point out that the term $C(1+\varepsilon^{-1})$ could be
slightly refined in $C(1+\varepsilon^{-1/2})$.  We will be sloppy in
carrying out the dependence on $\varepsilon$ because, as we have just
said, we will remove this dependence, and, at this stage, we only want
estimates independent of $n$.
\end{rmrk}

\paragraph{Second a priori estimate.}
We define $g_1:[0,T]\to H$ as
$g_1(t)=\gamma\partial_tw_n(t)-\partial_t\varphi_n(t)-\pi(\varphi_n(t))$.
Due to the first a priori estimate and the Lipschitz-continuity of
$\pi$\pier{,} we have that
\begin{equation*}
\norm{g_1}_{L^2(0,T;H)}\leq
C(1+\varepsilon^{-1}).  
\end{equation*}
We rewrite~\eqref{eq:phase-faedo-galerkin} as
\begin{equation*}
-(\Laplace\varphi_n(t),v)+(\beta_\varepsilon(\varphi_n(t)),v)=(g_1(t),v),
\end{equation*}
and we test with $v=-\Laplace\varphi_n(t)$\pier{:}
\begin{equation*}
  \norm{\Laplace\varphi_n(t)}_H^2-
  (\beta_\varepsilon(\varphi_n(t)),\Laplace\varphi_n(t))
  =-(g_1,\Laplace\varphi_n)
  \leq
  \norm{g_1(t)}_H\norm{\Laplace\varphi_n(t)}_H.
\end{equation*}
The second term is positive, because of
\begin{equation*}
  -(\beta_\varepsilon(\varphi_n),\Laplace\varphi_n)
  =-\int_\Omega\beta_\varepsilon(\varphi_n)\Laplace\varphi_n
  =\int_\Omega\beta_\varepsilon'(\varphi_n)|\nabla\varphi_n|^2\geq 0,
\end{equation*}
then yielding $\norm{\Laplace\varphi_n(t)}_H\leq\norm{g_1(t)}_H$.
Since $P_n(\beta_\varepsilon(\varphi_n))=P_n(g_1)+\Laplace\varphi_n$,
owing to elliptic regularity, we conclude that
\begin{equation}\label{eq:second-a-priori-estimate}
  \myboxed{
\norm{\varphi_n}_{
  L^2(0,T;W)}
+\norm{P_n(\beta_\varepsilon(\varphi_n))}_{
  L^2(0,T;H)}
\leq C(1+\varepsilon^{-1}).
}
\end{equation}

\paragraph{Third a priori estimate.}
We define $\eta,g_2:[0,T]\to V_n$ as
\begin{gather*}
  \eta(t):=\partial_t w_n(t)+\alpha\varphi_n(t)-\eta^*_n,\\[4mm]
  \begin{split}
     g_2(t)&:=(\alpha-l)\partial_t\varphi_n(t)-\alpha\kappa\Laplace\varphi_n(t)
     +\kappa\Laplace\eta^*_n\\
     &\,\,+\tau\Laplace\left(w_{0,n}+
     \alpha\int_0^t\varphi_n(s)\de s+t\eta^*_n\right)
     +f_n(t)
   \end{split}
\end{gather*}
\pier{for $t\in [0,T].$}
Thanks to equations~\eqref{eq:first-a-priori-estimate}
and~\eqref{eq:second-a-priori-estimate} we have that
\begin{equation*}
\norm{g_2}_{L^2(0,T;H)}+\norm{\eta}_{L^2(0,T;H)}\leq
C(1+\varepsilon^{-1}).
\end{equation*}
Moreover, equation~\eqref{eq:initial-data-hypothesis} implies
\begin{equation*}
  \norm{\eta(0)}_V=\norm{\vartheta_{0,n}+\alpha\varphi_{0,n}+\eta^*_n}_V
  \leq C.
\end{equation*}
We rewrite
equation~\eqref{eq:heat-faedo-galerkin} as
\begin{equation*}\label{eq:heat-rewritten-g2}
  \left(\partial_t\eta-\kappa\Laplace\eta-
  \tau\int_0^t\Laplace\eta(s)\de s
  +\rho\Sign_\varepsilon(\eta),v\right)
  =(g_2,v).
\end{equation*}
In view of equation~\eqref{eq:quasi-norma-def}, it is clear that
\begin{align*}
\left(\Laplace\eta(t),\int_0^t\Laplace\eta(s)\de s\right)
&=
\frac{1}{2}\cfrac{\de}{\de t}
\left\norm{\int_0^t\Laplace\eta(s)\de s\right}_H^2
,\\
-\int_\Omega\Sign_\varepsilon(\eta)\Laplace\eta
&=
\int_\Omega\nabla\Sign_\varepsilon(\eta)\cdot\nabla\eta
\geq 0
,\\
-\int_\Omega\partial_t\eta\Laplace\eta
&=
\cfrac{1}{2}
\cfrac{\de}{\de t}
\int_\Omega|\nabla\eta|^2.
\end{align*}
Thus, we test~\eqref{eq:heat-rewritten-g2} with $v=-\Laplace\eta(t)$
and we integrate over time finding
\begin{equation*}
\begin{aligned}
  \cfrac{1}{2}\norm{\nabla\eta(t)}_H^2
  &+\kappa\int_0^t\norm{\Laplace \eta}_H^2
  +\cfrac{\tau}{2}
  \left\norm{\int_0^t\Laplace\eta\de s\right}_H^2\\
  &\leq
  \cfrac{1}{2}
  \norm{\nabla\eta(0)}_H^2
  -\int_0^t (g_2,\Laplace\eta)\\
  &\leq
  C+
  \cfrac{\kappa}{2}
  \int_0^t\norm{\Laplace\eta}_H^2
  +\cfrac{1}{2\kappa}
  \int_0^t\norm{g_2}_H^2.
\end{aligned}
\end{equation*}
Hence, we infer
\begin{equation*}
\norm{\nabla\eta}_{L^\infty(0,T;H)}+\norm{\Laplace\eta}_{L^2(0,T;H)}\leq
C(1+\varepsilon^{-1}),
\end{equation*}
\marginnote{attenzione, ho aggiunto la parte sul
  gradiente}which, together with elliptic regularity, implies
$\norm{\eta}_{L^2(0,T;W)}\leq C(1+\varepsilon^{-1})$\pier{. Thus, it turns out that}
\begin{equation*}
\norm{\partial_tw_n}_{L^\infty(0,T;V)} +
\norm{\partial_tw_n}_{L^2(0,T;W)}\leq C(1+\varepsilon^{-1}).
\end{equation*}
Finally, as $\norm{w_n}_{L^\infty(0,T;H)}\leq C(1+\varepsilon^{-1})$
and $w_{0,n}\in W$, we conclude that \marginnote{ora serve $w_0\in W$
  e non solo in $H^2$}
\begin{equation}
  \myboxed{
\norm{w_n}_{W^{1,\infty}(0,T;V)\cap H^1(0,T;W)}
\leq C(1+\varepsilon^{-1}).
}
\end{equation}

\paragraph{Fourth a priori estimate.}
We define $g_3:[0,T]\to V_n$ as
\begin{equation*}
  g_3(t)=g_2(t)+\kappa\Laplace\eta(t)+\tau\int_0^t\Laplace\eta(s)\de s
  \pier{, \quadd t\in [0,T].}
\end{equation*}
Again, it holds true that $\norm{g_3}_{L^2(0,T;H)}\leq
C(1+\varepsilon^{-1})$ by comparison.  We rewrite
equation~\eqref{eq:heat-rewritten-g2} as
\begin{equation}\label{eq:heat-rewritten-g3}
(\eta_t+\rho\Sign_\varepsilon(\eta),v)=(g_3,v).
\end{equation}
Since $\frac{\de}{\de
  t}\norm{\eta}_{H,\varepsilon}=(\Sign_\varepsilon(\eta),\partial_t\eta)_H$,
we can test~\eqref{eq:heat-rewritten-g3} with $v=\partial_t\eta$
obtaining
\begin{equation*}
  \begin{aligned}
    \int_0^t\norm{\partial_t\eta}_H^2
    +\rho\norm{\eta(t)}_H
    &=\rho\norm{\eta(0)}_H
    +\int_0^t(g_3(s),\partial_t\eta(s))\de s\\
    &\leq \rho C+\cfrac{1}{2}\int_0^t\norm{\partial_t\eta}_H^2
    +\cfrac{1}{2}\int_0^t\norm{g_3}_H^2.
  \end{aligned}
\end{equation*}
Thus, $\norm{\partial_t\eta}_{L^2(0,T;H)}^2\leq
C(1+\rho+\varepsilon^{-1})$ and then, by comparison, we \pier{have that}
\begin{equation}
  \myboxed{
\norm{w_n}_{
  H^2(0,T;H)}
\leq C(1+\rho^{1/2}+\varepsilon^{-1}).
}
\end{equation}

\paragraph{Passage to the limit in the Faedo--Galerkin scheme.}
Making use of standard weak or weak* compactness results, possibly
taking a subsequence, we have that $(w_n,\varphi_n)$ converges in the
following topologies
\begin{align}
  \label{eq:heat-convergence-linfty}
  w_n\to w_\varepsilon
  \quadd
  &\text{ weakly in }
  H^2(0,T;H)\cap H^1(0,T;W),\\
  \label{eq:phase-convergence-l2}
  \varphi_n\to\varphi_\varepsilon
  \quadd
  &\text{ weakly in }
  H^1(0,T;H)\cap L^2(0,T;W),\\
  w_n\to w_\varepsilon
  \quadd
  &\text{ weakly* in }
  W^{1,\infty}(0,T;V),\\
  \label{eq:phase-convergence-linfty}
  \varphi_n\to\varphi_\varepsilon
  \quadd
  &\text{ weakly* in }
  L^\infty(0,T;V),
\end{align}
for a suitable pair $(w_\varepsilon,\varphi_\varepsilon)$.  This
implies, together with the generalized Ascoli theorem and the
Aubin--Lions theorem~\cite[Sec. 8, Cor. 4]{simon87},\marginnote{pagina?}
the following strong convergences
\begin{align}
  \label{eq:w-convergence-c1}
  w_n\to w_\varepsilon
  \quadd
  &\text{ in }
  H^1(0,T;V)\cap C^1([0,T];H),\\
  \label{eq:phase-convergence-c0}
  \varphi_n\to\varphi_\varepsilon
  \quadd
  &\text{ in }
  C^0([0,T];H)\cap L^2(0,T;V).
\end{align}
Hence, we have the \pier{convergences} 
\begin{align*}
  \Sign_\varepsilon(\partial_t w_n
  +\alpha\varphi_n-\eta_n^*)
  &\to
  \Sign_\varepsilon(\partial_t w_\varepsilon
  +\alpha\varphi_\varepsilon-\eta^*),
  \\
  \pi(\varphi_n)
  &\to
  \pi(\varphi_\varepsilon),
  \\
  \beta_\varepsilon(\varphi_n)
  &\to
  \beta_\varepsilon(\varphi_\varepsilon)
\end{align*}
\pier{in $C^0([0,T];H)$.}
We note that $\eta^*_n\to\eta^*$ and that the initial conditions hold true
\begin{equation*}
  \partial_tw_\varepsilon(0)=\vartheta_0
  ,\quadd
  w_\varepsilon(0)=w_0
  ,\quadd
  \varphi_\varepsilon(0)=\varphi_0.
\end{equation*}
Indeed, the property~\eqref{eq:proiezioni} implies the strong
convergence of the initial data and the target function $\eta^*$.  We
take $n,h\in\natr$ with $n>h$ and $v\in V_h\subset V_n$.  Since all
the involved terms converge, we take the limit as $n\to +\infty$ in
equations~\eqref{eq:heat-faedo-galerkin}
and~\eqref{eq:phase-faedo-galerkin}, obtaining
\begin{align*}
  &
  \begin{aligned}
    &(\partial_t^2w_\varepsilon+l\partial_t\varphi_\varepsilon
    -\kappa\Laplace \partial_t w_\varepsilon-\tau\Laplace w_\varepsilon
    \\
    &+\rho\Sign_\varepsilon(\partial_tw_\varepsilon+
    \alpha\varphi_\varepsilon-\eta^*),v)
    =(f,v)\quadd\forall v\in V_h,\text{ a.e. in }(0,T),
  \end{aligned}
  \\[2mm]
  &(\partial_t\varphi_\varepsilon
  -\Laplace \varphi_\varepsilon
  +\beta_\varepsilon(\varphi_\varepsilon)
  +\pi(\varphi_\varepsilon),v)
  =\gamma(\partial_t w_\varepsilon,v)
  \quadd\forall v\in V_h,\text{ a.e. in }(0,T).
\end{align*}
As $h$ is arbitrary, the above equations hold for all
$v\in\cup_{h=1}^{+\infty}V_h$. By density of $\cup_{h=1}^{+\infty}V_h$
in $H$, we find
\begin{align}\label{eq:heat-yosida}
  &
  \begin{aligned}
    &(\partial_t^2w_\varepsilon+l\partial_t\varphi_\varepsilon
    -\kappa\Laplace \partial_t w_\varepsilon-\tau\Laplace w_\varepsilon
    \\
    &+\rho\Sign_\varepsilon(\partial_tw_\varepsilon+
    \alpha\varphi_\varepsilon-\eta^*),v)
    =(f,v)\quadd\forall v\in H,\text{ a.e. in }(0,T),
  \end{aligned}
  \\[2mm]\label{eq:phase-yosida}
  &(\partial_t\varphi_\varepsilon
  -\Laplace \varphi_\varepsilon
  +\beta_\varepsilon(\varphi_\varepsilon)
  +\pi(\varphi_\varepsilon),v)
  =\gamma(\partial_t w_\varepsilon,v)
  \quadd\forall v\in H,\text{ a.e. in }(0,T).
\end{align}

\paragraph{Passage to the limit as \texorpdfstring{$\varepsilon\to
    0$}{epsilon to 0}.}
Let $\xi_\varepsilon:=\beta_\varepsilon(\varphi_\varepsilon)$ and
$\sigma_\varepsilon:=\Sign_\varepsilon(\partial_t w_\varepsilon
+\alpha\varphi_\varepsilon-\eta^*)$.  We now review the a priori
estimates in order to remove the dependence on $\varepsilon$.  All
calculations are still working if we make the\marginnote{boh? evitiamo overful hbox}
 substitution
\begin{equation*}
  (w_\varepsilon,\varphi_\varepsilon,\xi_\varepsilon,\sigma_\varepsilon)
  \quad\longmapsto\quad
  (w_n,\varphi_n,\beta_\varepsilon(\varphi_n),\Sign_\varepsilon(\partial_tw_n+\alpha\varphi_n-\eta^*_n)).
\end{equation*}
By Remark~\ref{rmrk:epsilon-dependence}, the dependence on
$\varepsilon$ is only given by equation~\eqref{eq:annoying}.  We
observe that, owing to~\eqref{eq:standard-yosida-for-beta},
\begin{equation*}
\norm{\widehat\beta_\varepsilon(\varphi_0)}_{L^1(\Omega)} \leq\norm{\widehat\beta(\varphi_0)}_{L^1(\Omega)},
\end{equation*}
and we had just made the first a priori estimate independent of
$\varepsilon$:
\begin{equation}
  \label{eq:first-a-priori-estimate-without-epsilon}
  \begin{split}
  &\norm{w_\varepsilon}_{
    W^{1,\infty}(0,T;H)\cap
    H^1(0,T;V)
  }
  +\norm{\varphi_\varepsilon}_{
    H^1(0,T;H)\cap
    L^\infty(0,T;V)}
  \\&
  +\rho\int_0^T
  \norm{\partial_tw_\varepsilon+\alpha\varphi_\varepsilon+\eta^*}_{H,\varepsilon}
  +\norm{\widehat\beta_\varepsilon(\varphi_\varepsilon)}_{
    L^\infty(0,T;L^1(\Omega))}
  \leq C.
  \end{split}
\end{equation}
Having removed the dependence on $\varepsilon$ in the first estimate,
all other estimates can be replicated obtaining
\begin{align}
  &\norm{\varphi_\varepsilon}_{L^2(0,T;W)}
  +\norm{\xi_\varepsilon}_{L^2(0,T;H)}
  \leq C,\\
  &\norm{w_\varepsilon}_{W^{1,\infty}(0,T;V)\cap H^1(0,T;W)}
  \leq C,\\
  \label{eq:fourth-a-priori-estimate-without-epsilon}
  &\norm{w_\varepsilon}_{H^2(0,T;H)}
  \leq C(1+\rho^{1/2}).
\end{align}
Moreover, because of the definition of the $\Sign$ operator,
$\sigma_\varepsilon(t)$ is bounded, uniformly with respect to $t$ and
$\varepsilon$, i.e.,
\begin{equation}
  \label{eq:sigma-estimate}
\norm{\sigma_\varepsilon}_{L^\infty(0,T;H)}
\leq 1.
\end{equation}

We are now able to take the limit as $\varepsilon\to 0$ using the same
compactness argument as before.  There exists a quadruplet
$(w,\varphi,\xi,\sigma)$ such that (a subsequence of)
$(w_\varepsilon,\varphi_\varepsilon,\xi_\varepsilon,\sigma_\varepsilon)$
converges to $(w,\varphi,\xi,\sigma)$ in the same topologies as
before.
More precisely for
$\partial_tw_\varepsilon+\alpha\varphi_\varepsilon$,
$\xi_\varepsilon$, and $\sigma_\varepsilon$ we have that
\begin{align}
  \partial_tw_\varepsilon+\alpha\varphi_\varepsilon
  &\to
  \partial_tw+\alpha\varphi
  \quadd
  &&\text{ in }C^0([0,T],H),
  \label{eq:convergence-eta-c0}
  \\
  \xi_\varepsilon
  &\to\xi
  \quadd
  &&\text{ weakly in }
  L^2(0,T;H),
  \\
  \label{eq:sigma-convergence}
  \sigma_\varepsilon&\to\sigma
  \quadd
  &&\text{ weakly in }
  L^2(0,T;H).
\end{align}
We take the limit in equation~\eqref{eq:heat-yosida}
and~\eqref{eq:phase-yosida} obtaining~\eqref{eq:heat}
and~\eqref{eq:phase}, respectively.  Since $\varphi_\varepsilon$ and
$\partial_tw_\varepsilon$ converge strongly in $L^2(0,T;H)=L^2(Q)$
\marginnote{va bene l'argomento scritto cos\`i?}  and
$\xi_\varepsilon$ and $\sigma_\varepsilon$ converge weakly, we deduce
\begin{align*}
\lim_{\varepsilon\to 0}
\int_{Q} \xi_\epsilon\,\varphi_\epsilon&=
\int_{Q} \xi\,\varphi,\\
\lim_{\varepsilon\to 0}
\int_{Q} \sigma_\epsilon\,(\partial_tw_\epsilon+\alpha\varphi_\epsilon-\eta^*)&=
\int_{Q} \sigma\,(\partial_tw+\alpha\varphi-\eta^*).
\end{align*}
Hence, by~\cite[Prop.~2.2,~p.~38]{barbu10}
\marginnote{biblio}
we have that
\begin{equation*}
\xi\in\beta(\varphi)
\quadd
\text{ and }
\quadd
\sigma\in
\Sign(\partial_tw+\alpha\varphi-\eta^*),
\end{equation*}
almost everywhere, and the proof of the existence of the solutions is
complete.

To conclude the proof of Theorem~\ref{thrm:existence} we need to
prove~\eqref{eq:existence-bound1} and~\eqref{eq:existence-bound2}.
Owing to the lower semi-continuity of the norms, the Fatou lemma, and
part 4 of Proposition~\ref{prop:yosida}, we can take the inferior
limit as $\varepsilon\to0$
in~\eqref{eq:first-a-priori-estimate-without-epsilon}--\eqref{eq:sigma-estimate}
deducing
\begin{align}
  &
  \begin{aligned}
     &\norm{w}_{ W^{1,\infty}(0,T;V)\cap H^1(0,T;W)}
     +\norm{\varphi}_{H^1(0,T;H)\cap L^\infty(0,T;V)\cap L^2(0,T;W)}
     \\&
     +\rho\norm{w_t+\alpha\varphi-\eta^*}_{L^1(0,T;H)}+\norm{\widehat\beta(\varphi)}_{L^\infty(0,T;L^1(\Omega))}
     \\&
     +\norm{\xi}_{L^2(0,T;H)}
     +\norm{\sigma}_{L^\infty(0,T;H)}
     \leq C,
  \end{aligned}
  \\[2mm]&
  \norm{w}_{H^2(0,T;H)}  \leq C(1+\rho^{1/2}).
\end{align}

\subsection{Existence of solutions for \problem{B}}
The proof of Theorem~\ref{thrm:existenceB} is similar to the proof of
the previous one.
For this reason \pier{the} steps with few relevance will be omitted.
\paragraph{Faedo--Galerkin approximation.}
Most of the first paragraph of the previous subsection can be
replicated {\em verbatim}.
The first difference appears in
equation~\eqref{eq:proj-target-function}; in this case, we project the
target function~$\varphi^*$
\begin{equation}
\varphi^*_{n}:=P_n \varphi^*.
\end{equation}
The approximated problem is to find two functions $w_n\in
C^2([0,T];V_n)$ and $\varphi_n\in C^1([0,T];V_n)$, such
that\marginnote{formulazione variazionale}
\begin{align}
  \label{eq:heatB-faedo-galerkin}
  &
  \begin{aligned}
    &
    (\partial_t^2w_n+l\partial_t\varphi_n
    -\kappa\Laplace \partial_t w_n-\tau\Laplace w_n,v)
    =(f_n,v),
    \\&
    \quadd\forall v\in V_n,\text{ in }[0,T],
  \end{aligned}
  \\[2mm]
  \label{eq:phaseB-faedo-galerkin}
  &
  \begin{aligned}
  &(\partial_t\varphi_n
  -\Laplace \varphi_n
  +\beta_\varepsilon(\varphi_n)
  +\pi(\varphi_n)
  +\rho\Sign_\varepsilon(\varphi_n-\varphi^*_n),v)
  \\&
  \quadd=\gamma(\partial_t w_n,v),
  \quadd\forall v\in V_n,\text{ in }[0,T],
  \end{aligned}
  \\[2mm]
  \label{eq:initial-dataB-faedo-galerkin}
  &\partial_t w_n(0)=\vartheta_{0,n}
,\quadd
w_n(0)=w_{n,0}
,\quadd
\varphi_n(0)=\varphi_{n,0}.  
\end{align}
Like for \problem{A}, the system above has a unique solution, given by
Cauchy--Lipschitz theorem.

\paragraph{First a priori estimate.}
We take $v=\partial_tw_n+l\varphi_n$ and $v=\kappa
l^2(\varphi_n-\varphi_n^*)$ in
equations~\eqref{eq:heatB-faedo-galerkin}
and~\eqref{eq:phaseB-faedo-galerkin}, respectively.
We sum up, integrate between $0$ and $t$, and, \pier{taking~\eqref{eq:quasi-norma}} into
account, reorder obtaining
\begin{equation*}
  \begin{aligned}
    &
    \frac{1}{2}\norm{\partial_tw_n+l\varphi_n}_H^2
    +\kappa\int_0^t\norm{\partial_t\nabla w_n}_H^2
    +\frac{\tau}{2}\norm{\nabla w_n}_H^2
    +\kappa l\int_0^{t}(\partial_t\nabla w_n,\nabla\varphi_n)
    \\&
    +\frac{\kappa l^2}{2}\norm{\varphi_n}_H^2
    +\kappa l^2\int_{0}^{t}\norm{\nabla\varphi_n}_H^2
    +\kappa l^2\int_0^t(\nabla\varphi_n,\nabla\varphi_n^*)
    \\&
    +\kappa l^2
    \int_0^t
    (\beta_\varepsilon(\varphi_n)-\beta_\varepsilon(\varphi_n^*),
    \varphi_n-\varphi_n^*)
    +\rho\kappa l^2
    \int_0^t\norm{\varphi_n-\varphi_n^*}_{H,\varepsilon}
    \\&
    \leq
    \frac{1}{2}\norm{\partial_tw_{0,n}+l\varphi_{0,n}}_H^2
    +\frac{\tau}{2}\norm{\nabla w_{0,n}}_H^2
    +\frac{\kappa l^2}{2}\norm{\varphi_{0,n}}_H^2
    \\&
    -\tau l
    \int_0^t(\nabla w_n,\nabla\varphi_n)
    +\int_0^t(f_n,\partial_tw_n+l\varphi_n)
    \\&
    -\kappa l^2
    \int_0^t(\pi(\varphi_n),\varphi_n-\varphi_n^*)
    +\gamma\kappa l^2
    \int_0^t(\partial_tw_n,\varphi_n-\varphi_n^*)
    \\&
    -\kappa l^2
    \int_0^t(\beta_\varepsilon(\varphi_n^*),\varphi_n-\varphi_n^*)
    +\kappa l^2(\varphi_n,\varphi_n^*)
    .
  \end{aligned}
\end{equation*}
Regarding the left-hand side of the inequality above, using the Young
inequality, it can be easily proven~that
\begin{align*}
  &
  \frac{1}{2}\norm{\partial_tw_n+l\varphi_n}_H^2
  +\frac{\kappa l^2}{2}\norm{\varphi_n}_H^2
  \geq
  \cfrac{l^2\kappa}{4}\norm{\varphi_n}_H^2
  +\frac{\kappa}{2(\kappa+2)}\norm{\partial_tw_n}_H^2,
  \\&
  \norm{\partial_t\nabla w_n}_H^2
  +l(\partial_t\nabla w_n,\nabla\varphi_n)
  +l^2\norm{\nabla\varphi_n}_H^2
  \geq
  \frac{1}{2}(\norm{\partial_t\nabla w_n}_H^2
  +l^2\norm{\nabla\varphi_n}_H^2).
\end{align*}
The non-constant terms in the right-hand side of the inequality are
easily controlled via the Young inequality\pier{: we have that}
\begin{align*}
  &
  \int_0^t(\pi(\varphi_n),\varphi_n-\varphi_n^*)
  \leq
  C\int_0^t(\norm{\varphi_n}_H^2+1),
  \\&
  \tau l\int_0^t(\nabla w_n,\nabla\varphi_n)
  \leq
  \frac{\kappa l^2}{4}\int_0^t\norm{\nabla\varphi_n}_H^2
  +\frac{\tau}{\kappa}\int_0^t\norm{\nabla w_n}_H^2,
  \\&
  \int_0^t(f_n,\partial_tw_n-l\varphi_n)
  \leq C+\int_0^t(\norm{\partial_tw_n}_H^2+l^2\norm{\varphi_n}_H^2),
  \\&
  \int_0^t(\partial_tw_n,\varphi_n-\varphi_n^*),
  \leq
  C+\frac{1}{2}\int_0^t(\norm{\partial_tw_n}_H^2+\norm{\varphi_n}_H^2)
\end{align*}
\pier{as well as} 
\begin{align*} 
  &\kappa l^2(\varphi_n,\varphi_n^*)
  \leq\frac{\kappa l^2}{8}\norm{\varphi_n}_H^2
  +2\kappa l^2\norm{\varphi^*_n}_H^2,
  \\&
  \int_0^t(\beta_\varepsilon(\varphi_n^*),\varphi_n-\varphi_n^*)
  \leq
  \frac 1 2 \int_0^t\norm{\beta_\varepsilon(\varphi_n^*)}_H^2
  +\int_0^t\norm{\varphi_n}_H^2+T\norm{\varphi^*_n}_H^2
  .
\end{align*}
We notice that (cf.~\eqref{eq:beta-lipschitz}) \marginnote{tra un po' servir\`a $\beta(\varphi^*)\in
  H$}
\begin{equation}
  \label{eq:annoyingB}
  \int_0^t\norm{\beta_\varepsilon(\varphi_n^*)}_H^2
  \leq T\varepsilon^{-2}\norm{\varphi_n^*}_H^2\leq C\varepsilon^{-2}.
\end{equation}
We can use Gronwall lemma obtaining
\begin{equation}
  \label{eq:first-a-priori-estimateB}
  \myboxed{
  \begin{aligned}
    &
    \norm{w_n}_{W^{1,\infty}(0,T;H)\cap H^1(0,T;V)}
    +\norm{\varphi_n}_{L^\infty(0,T;H)\cap L^2(0,T;V)}
    \\&
    +\rho\int_0^T\norm{\varphi_n-\varphi_n^*}_{H,\varepsilon}
    \leq C(1+\varepsilon^{-1}).
  \end{aligned}
  }
\end{equation}
The dependence on $\varepsilon$ is given only by
equation~\eqref{eq:annoyingB}.
However, Remark~\ref{rmrk:epsilon-dependence} still remains valid, so
the dependence on $\varepsilon$ in
equation~\eqref{eq:first-a-priori-estimateB} will be later removed.
\paragraph{Second a priori estimate.}
We take $v=\partial_t\varphi_n$ in~\eqref{eq:phaseB-faedo-galerkin}
and integrate over $(0,t)$, obtaining
\begin{equation*}
  \begin{aligned}
    &
    \int_0^t\norm{\partial_t\varphi_n}_H^2
    +\norm{\nabla\varphi_n}_H^2
    +\norm{\widehat\beta_\varepsilon(\varphi_n)}_{L^1(\Omega)}
    +\rho\norm{\varphi_n-\varphi_n^*}_{H,\varepsilon}
    \\&
    =\norm{\nabla\varphi_{n,0}}_H^2
    +\norm{\widehat\beta_\varepsilon(\varphi_{n,0})}_{L^1(\Omega)}
    +\rho\norm{\varphi_{n,0}-\varphi_n^*}_{H,\varepsilon}
    \\&
    \quad +\int_0^t(-\pi(\varphi_n)+\gamma\partial_t
    w_n,\partial_t\varphi_n)
    \\&
    \leq
    C(1+\rho+\varepsilon^{-1})
    +\int_0^t(-\pi(\varphi_n)+\gamma\partial_t
    w_n,\partial_t\varphi_n)    
    .
  \end{aligned}
\end{equation*}
The last integral can be estimated using Young inequality, the
Lipschitz-conti\-nuity of $\pi$, and the first a priori estimate\pier{:}
\begin{equation*}
  \begin{aligned}
    &
  \int_0^t(-\pi(\varphi_n)+\gamma\partial_t w_n,\partial_t\varphi_n)
  \\&
  \leq
  \frac{1}{2}
  \norm{-\pi(\varphi_n)+\gamma\partial_t w_n}_{L^2(0,T;H)}^2
  +\frac{1}{2}\int_0^t\norm{\partial_t\varphi_n}_H^2
  \\&
  \leq
  C(1+\varepsilon^{-1})
  +\frac{1}{2}\int_0^t\norm{\partial_t\varphi_n}_H^2.
  \end{aligned}
\end{equation*}
Hence, we deduce
\begin{equation}
  \myboxed{
  \begin{aligned}
    &
  \norm{\varphi_n}_{H^1(0,T;H)\cap L^\infty(0,T;V)}
  +\norm{\widehat\beta_\varepsilon(\varphi_n)}_{L^\infty(0,T;L^1)}
  \\&
  +\rho\sup_{t\in(0,T)}\norm{\varphi_n(t)-\varphi_n^*}_{H,\varepsilon}
  \leq C(1+\rho^{1/2}+\varepsilon^{-2}).
  \end{aligned}
  }
  \label{eq:second-a-priori-estimateB}
\end{equation}

\paragraph{Third a priori estimate.}
We define $g_1:[0,T]\to H$ as
\begin{equation*}
  g_1=-\partial_t\varphi_n-\pi(\varphi_n)+\gamma\partial_t w_n
  .
\end{equation*}
Clearly\pier{, it holds that} $\norm{g_1}_{L^2(0,T;H)}\leq C(1+\rho+\varepsilon^{-1})$,
and we can rewrite equation~\eqref{eq:phaseB-faedo-galerkin} as
\begin{equation*}
  (-\Laplace\varphi_n+\beta_\varepsilon(\varphi_n)
  +\rho\Sign_\varepsilon(\varphi_n-\varphi_n^*),v)
  =(g_1,v).
\end{equation*}
We take $v=-\Laplace\varphi_n$ and \pier{obtain}
\begin{equation*}
  \begin{aligned}
  &\norm{\Laplace\varphi_n}_H^2
  +\int_\Omega\beta_\varepsilon'(\varphi_n)|\nabla\varphi_n|^2
  +\rho(\nabla\Sign_\varepsilon(\varphi_n-\varphi_n^*),
  \nabla(\varphi_n-\varphi_n^*))
  \\&
  =(g_1,\Laplace\varphi_n)
  +\rho(\Sign_\varepsilon(\varphi_n-\varphi_n^*),\Laplace\varphi_n^*)
  \\&
  \leq\frac 1 2 \norm{g_1}_H^2
  +\frac 1 2 \norm{\Laplace\varphi_n}_H^2
  +\rho\norm{\Laplace\varphi_n^*}_H.
  \end{aligned}
\end{equation*}
Since the integral in the left-hand side of the equation above is non
negative, we can forget about it.
We can also cancel the third term, because
\begin{equation*}
  (\nabla\Sign_\varepsilon(\varphi_n-\varphi_n^*),\nabla(\varphi_n-\varphi_n^*))
  =\frac{\norm{\nabla(\varphi_n-\varphi_n^*)}_H^2}
  {\max\{\varepsilon,\norm{\varphi_n-\varphi_n^*}_H\}}
  \geq0.
\end{equation*}
We integrate over time, obtaining
\begin{equation*}
  \norm{\Laplace\varphi_n}_{L^2(0,T;H)}\leq
  C(1+\rho^{1/2}+\varepsilon^{-1}).
\end{equation*}
By comparison and the elliptic regularity, we deduce that
\begin{equation}
  \myboxed{
    \begin{aligned}
      &
    \norm{\varphi_n}_{L^2(0,T;W)}
    +\norm{\beta_\varepsilon(\varphi_n)
      +\rho\Sign_\varepsilon(\varphi_n-\varphi_n^*)}_{L^2(0,T;H)}
    \\&
    \leq C(1+\rho^{1/2}
    +\varepsilon^{-1}).
    \end{aligned}
  }
\end{equation}

\paragraph{Fourth a priori estimate.}
We define $\eta_n:[0,T]\to V_n$ as
\begin{gather*}
  \eta_n(t):=\partial_t w_n(t)+l\varphi_n(t) \pier{, \quadd t\in [0,T].}
\end{gather*}
Thanks to the estimate~\eqref{eq:first-a-priori-estimateB} we have that
\begin{equation*}
\norm{\eta_n}_{L^\infty(0,T;H)}\leq C(1+\varepsilon^{-1}).
\end{equation*}
Moreover, equation~\eqref{eq:initial-data-hypothesis} implies
\begin{equation*}
  \norm{\eta_n(0)}_V=\norm{\vartheta_{0,n}+l\varphi_{0,n}}_V \leq C.
\end{equation*}
We rewrite equation~\eqref{eq:heatB-faedo-galerkin} as
\begin{equation*}\label{eq:heat-rewrittenB}
  \left(\partial_t\eta_n-\kappa\partial_t\Laplace w_n-
  \tau\Laplace w_n,v\right)
  =(f_n,v).
\end{equation*}
Thus, we test the equation above with $v=-\Laplace\eta_n(t)$
and we integrate over time finding
\begin{equation*}
\begin{aligned}
  &\cfrac{1}{2}\norm{\nabla\eta_n(t)}_H^2
  +\kappa\int_0^t\norm{\partial_t\Laplace w_n}_H^2
  +\cfrac{\tau}{2}
  \norm{\Laplace w_n}_H^2
  \\&
  \leq
  C+\kappa l\int_0^t(\partial_t\Laplace w_n,\Laplace\varphi_n)
  +\tau l\int_0^t(\Laplace w_n,\Laplace\varphi_n)
  -\int_0^t (f_n,\Laplace\eta_n).
\end{aligned}
\end{equation*}
The terms in the right-hand side of the inequality above can be
controlled easily
\begin{align*}
  &
  \kappa l\int_0^t(\partial_t\Laplace w_n,\Laplace\varphi_n)
  \leq
  \int_0^t\frac{\kappa}{2}\norm{\partial_t\Laplace w_n}_H^2
  +\int_0^t\frac{l^2}{2\kappa}\norm{\Laplace \varphi_n}_H^2,
  \\&
  \int_0^t(\Laplace w_n,\Laplace\varphi_n)
  \leq
  \frac{1}{2}\int_0^t(\norm{\Laplace w_n}_H^2+\norm{\Laplace\varphi_n}_H^2),
  \\&
  \int_0^t (f_n,\Laplace\eta_n)
  \leq
  C+
  \frac{\kappa}{4}\int_0^t
  \norm{\partial_t\Laplace w_n}_H^2.
\end{align*}
Hence, we infer
\begin{equation*}
\norm{\nabla\eta_n}_{L^\infty(0,T;H)}+\norm{\partial_t\Laplace w_n}_{L^2(0,T;H)}\leq
C(1+\rho^{1/2}+\varepsilon^{-1}),
\end{equation*}
\marginnote{attenzione, ho aggiunto la parte sul gradiente}which,
together with elliptic regularity and the
estimate~\eqref{eq:second-a-priori-estimateB}, implies
\begin{equation*}
\norm{\partial_tw_n}_{L^\infty(0,T;V)} +
\norm{\partial_tw_n}_{L^2(0,T;W)}\leq C(1+\rho^{1/2}+\varepsilon^{-1}).
\end{equation*}
Finally, as $w_{0,n}$ is bounded in $W$, we conclude that \marginnote{ora serve
  $w_0\in W$ e non solo in $H^2$}
\begin{equation}
  \myboxed{
\norm{w_n}_{W^{1,\infty}(0,T;V)\cap H^1(0,T;W)}
\leq C(1+\rho^{1/2}+\varepsilon^{-1}).
}
\end{equation}
By comparison in~\eqref{eq:heatB-faedo-galerkin}, it \pier{follows} that
\begin{equation}
  \myboxed{
    \norm{w_n}_{H^2(0,T;H)\cap W^{1,\infty}(0,T;V)\cap H^1(0,T;W)}
    \leq C(1+\rho^{1/2}+\varepsilon^{-1}).}
\end{equation}

\paragraph{Passage to the limit.}
The arguments used in the previous proof for the passage to the limit
in the Faedo--Galerkin scheme and as $\varepsilon\to 0$ are still
perfectly working.
First, passing to the limit with respect to $n$
in~\eqref{eq:heatB-faedo-galerkin}--\eqref{eq:initial-dataB-faedo-galerkin},
we find a pair $(w_\varepsilon,\varphi_\varepsilon)$ satisfying
\begin{align}
  \label{eq:heatB-yosida}
  &
  \begin{aligned}
    &
    (\partial_t^2w_\varepsilon+l\partial_t\varphi_\varepsilon
    -\kappa\Laplace \partial_t w_\varepsilon-\tau\Laplace w_\varepsilon,v)
    =(f,v),
    \\&
    \quadd\forall v\in H,\text{ a.e.\ in }[0,T],
  \end{aligned}
  \\[2mm]
  \label{eq:phaseB-yosida}
  &
  \begin{aligned}
  &(\partial_t\varphi_\varepsilon
  -\Laplace \varphi_\varepsilon
  +\xi_\varepsilon
  +\pi(\varphi_\varepsilon)
  +\rho\sigma_\varepsilon
  ,v)
  =\gamma(\partial_t w_\varepsilon,v),
  \\&
  \quadd\forall v\in H,\text{ a.e.\ in }[0,T],
  \end{aligned}
  \\[2mm]
  &\partial_t w_\varepsilon(0)=\vartheta_{0}
,\quadd
w_\varepsilon(0)=w_{0}
,\quadd
\varphi_\varepsilon(0)=\varphi_{0},
\end{align}
where
\begin{align}
  \label{eq:xi-and-sigma-epsilonB}
  &
  \xi_\varepsilon:=\beta_\varepsilon(\varphi_\varepsilon)
  \quadd\text{ and }\quadd
  \sigma_\varepsilon:=\Sign_\varepsilon(\varphi_\varepsilon-\varphi^*).
\end{align}

We want to replicate the a priori estimate for this new setting.
As we have already pointed out, the only dependence on $\varepsilon$
are given via the estimate~\eqref{eq:annoyingB}.
We observe that (cf.~\eqref{eq:standard-yosida-for-beta})
\begin{equation}
  \norm{\beta_\varepsilon(\varphi^*)}_H\leq\norm{\pier{\beta^\circ}(\varphi^*)}_H,
\end{equation}
thus estimate~\eqref{eq:first-a-priori-estimateB} improves to
\begin{equation}
  \begin{aligned}
    &
    \norm{w_\varepsilon}_{W^{1,\infty}(0,T;H)\cap H^1(0,T;V)}
    +\norm{\varphi_\varepsilon}_{L^\infty(0,T;H)\cap L^2(0,T;V)}
    \\&
    +\rho\int_0^T\norm{\varphi_\varepsilon-\varphi^*}_{H,\varepsilon}
    \leq C,
  \end{aligned}
  \label{eq:first-a-priori-estimate-without-epsilonB}
\end{equation}
and we have just made the first a priori estimate independent of
$\varepsilon$.
Having removed the dependence on $\varepsilon$ in the first estimate,
all other estimates can be replicated obtaining
\begin{align}
  &
  \begin{aligned}
    &
    \norm{\varphi_\varepsilon}_{H^1(0,T;H)\cap L^\infty(0,T;V)}
    +\norm{\widehat\beta_\varepsilon(\varphi_\varepsilon)}_{L^\infty(0,T;L^1)}
    \\&
    +\rho\sup_{t\in(0,T)}\norm{\varphi_\varepsilon(t)-\varphi^*}_{H,\varepsilon}
    \leq C(1+\rho^{1/2}),
  \end{aligned}
  \\
  &
  \begin{aligned}
    &
    \norm{\varphi_\varepsilon}_{L^2(0,T;W)}
    +\norm{\xi_\varepsilon
      +\rho\sigma_\varepsilon}_{L^2(0,T;H)}
    \leq C(1+\rho^{1/2}),
  \end{aligned}
  \\&
  \begin{aligned}
    &
  \norm{w_\varepsilon}_{H^2(0,T;H)\cap W^{1,\infty}(0,T;V)\cap H^1(0,T;W)}
  \leq C(1+\rho^{1/2}).
  \end{aligned}
  \label{eq:fourth-a-priori-estimate-without-epsilonB}
\end{align}
Moreover, because of the definition of the $\Sign$ operator,
$\sigma_\varepsilon(t)$ is bounded, uniformly with respect to $t$ and
$\varepsilon$, i.e.,
\begin{equation}
  \label{eq:sigma-estimateB}
\norm{\sigma_\varepsilon}_{L^\infty(0,T;H)}
\leq 1.
\end{equation}
By comparison we have that $\norm{\xi_\varepsilon}_{L^2(0,T;H)}\leq
C(1+\rho)$.
Like for \problem{A}, we can use standard compactness
results and extract a subsequence such that
\begin{equation}
  (w_\varepsilon,\varphi_\varepsilon,\xi_\varepsilon,\sigma_\varepsilon)
  \to
  (w,\varphi,\xi,\sigma)
\end{equation}
in a \pier{suitable} topology for the
spaces~\eqref{eq:regw}--\eqref{eq:regsigma} where the solutions are
set.
Moreover, we have that
\begin{equation}
  \varphi_\varepsilon\to\varphi
  \quadd
  \text{ in }
  C^0([0,T];H).
  \label{eq:phase-convergence-c0B}
\end{equation}
We conclude that \problem{B} admits a solution and that~\eqref{eq:existence-bound1B}--\eqref{eq:existence-bound2B}
hold, using the same argument used in \problem{A}.

The uniqueness of the solution follows immediately from
Theorem~\ref{thrm:continuous-dependenceB}, and it is treated in
Section~\ref{continuous-dependence}.

\section{Further regularity}
This section contains the proofs of
Theorems~\ref{thrm:further-regularity}
and~\ref{thrm:further-regularityB}.
For the sake of clarity, both proofs can be divided in two parts.
In the former part we use the notations of the Faedo--Galerkin scheme
to prove that the limit function $\varphi_\varepsilon$ is more regular
(in particular $\varphi_\varepsilon\in H^1(0,T;V)$).
In the latter part, \marginnote{ho scritto subsubsection} the desired
estimates~\eqref{eq:further-regularity}
and~\eqref{eq:further-regularityB} are shown for the approximate
solution $\varphi_\varepsilon$.
Using the usual compactness argument and the lower semi-continuity
\marginnote{\`e giusto?} of the norm, the
estimates~\eqref{eq:further-regularity}
and~\eqref{eq:further-regularityB} will follow automatically.
In the present proofs, $C$ still denotes a positive constant independent of
$\rho$ and $\varepsilon$.

\subsection{Further regularity for \problem{A}}
We now prove Theorem~\ref{thrm:further-regularity}.
We consider equation~\eqref{eq:phase-faedo-galerkin} and note that: 1)
the functions $\Laplace\varphi_n$ and $\partial_tw_n$ are derivable
and their derivatives are $\Laplace\partial_t\varphi_n$ and
$\partial_t^2w_n$, respectively; 2) for all $v\in V_n$, the functions
\begin{equation*}
  t\mapsto(\beta_\varepsilon(\varphi_n(t)),v)
  \quadd\text{ and }\quadd
  t\mapsto(\pi(\varphi_n(t)),v)  
\end{equation*}
are Lipschitz-continuous, thus derivable a.e.\ in $(0,T)$ with
derivative
\begin{equation*}
  (\beta_\varepsilon'(\varphi_n(t))\partial_t\varphi_n(t),v)
  \quadd\text{ and }\quadd
  (\pi'(\varphi_n(t))\partial_t\varphi_n(t),v),
\end{equation*}
respectively.  Hence, $\partial_t\varphi_n$ is Lipschitz-continuous by
comparison and we can derive (in weak sense)
equation~\eqref{eq:phase-faedo-galerkin} obtaining
\begin{equation}\label{eq:further-regularity-faedo-galerkin}
  (\partial_t^2\varphi_n-\Laplace\partial_t\varphi_n
  +\beta_\varepsilon'(\varphi_n)\partial_t\varphi_n,v)=(g_4,v),
  \quadd\forall v\in V_n,\text{ a.e. in }(0,T),
\end{equation}
where
\begin{equation*}
  g_4:=-\pi'(\varphi_\varepsilon)\partial_t\varphi_n+
  \gamma\partial_t^2w_n
  .
\end{equation*}
Clearly $\norm{g_4}_{L^2(0,T;H)}\leq
C(1+\varepsilon^{-1}+\rho^{1/2})$, as $\pi'$ is bounded. We take
$v=\partial_t \varphi_\varepsilon$
in equation~\eqref{eq:further-regularity-faedo-galerkin} and we integrate over
$(0,t)$ obtaining
\begin{equation}
  \label{eq:further-regularity-big-estimate}
  \begin{split}
    &\cfrac{1}{2}
    \norm{\partial_t\varphi_n(t)}_H^2
    +\int_0^t
    \norm{\nabla\partial_t\varphi_n}_H^2
    +\int_{Q_t}
    \beta_\varepsilon'(\varphi_n)
    |\partial_t\varphi_n|^2
    \\&
    =\cfrac{1}{2}
    \norm{\partial_t\varphi_n(0)}_H^2
    +\int_0^t
    (g_4,\partial_t\varphi_n)
    \\&
    \leq
    \cfrac{1}{2}
    \norm{\partial_t\varphi_n(0)}_H^2
    +\cfrac{1}{2}
    \int_0^t
    \norm{g_4}_H^2
    +\cfrac{1}{2}
    \int_0^t
    \norm{\partial_t\varphi_n}_H^2
    \\&
    \leq
    \cfrac{1}{2}
    \norm{\partial_t\varphi_n(0)}_H^2
    +\cfrac{1}{2}
    \int_0^t
    \norm{\partial_t\varphi_n}_H^2
    +C(1+\varepsilon^{-2}+\rho).
  \end{split}
\end{equation}
Since $\beta_\varepsilon$ is monotone, we have that
$\beta_\varepsilon'\geq 0$ implies
$\int_{Q_t}\beta_\varepsilon'(\varphi_n)|\partial_t\varphi_n|^2\geq0$.
At this point, we want to use the Gronwall lemma to control
$\norm{\partial_t\varphi_n}_{L^2(0,T;H)}$.  The most delicate part is
to find a bound on $\norm{\partial_t\varphi_n(0)}_H$.  Using again
equation~\eqref{eq:phase-faedo-galerkin} we compute \marginnote{forse
  da richiamare meglio}
\begin{equation}\label{eq:further-regularity-initial-data}
  \begin{aligned}
    \norm{\partial_t\varphi_n(0)}_H&=
    \norm{\gamma\vartheta_{0,n}+\Laplace\varphi_{0,n}
      -P_n(\beta_\varepsilon(\varphi_{0,n}))
      -P_n(\pi(\varphi_{0,n}))}_H\\
    &\leq\gamma\norm{\vartheta_{0,n}}_H+\norm{\Laplace\varphi_{0,n}}_H
      +\norm{P_n(\pi(\varphi_{0,n}))}_H
    +\norm{P_n(\beta_\varepsilon(\varphi_{0,n}))}_H\\
    &\leq\gamma\norm{\vartheta_{0}}_H+\norm{\Laplace\varphi_{0}}_H
    +C\norm{\varphi_0}_H+\varepsilon^{-1}\norm{\varphi_0}_H
    \leq C(1+\varepsilon^{-1}) \pier{,}
  \end{aligned}
\end{equation}
where the fact that $\beta_\varepsilon$ is
$\varepsilon^{-1}$-Lipschitz-continous and the hypothesis
$\varphi_{0}\in W$ have been taken into account.  We incidentally note
that we have not yet used the hypothesis $\beta^\circ(\varphi_0)\in
H$.  Owing to the Gronwall lemma, we obtain
\begin{equation}\label{eq:further-regularity-epsilon}
\norm{\varphi_n}_{
  W^{1,\infty}(0,T;H)
  \cap
  H^1(0,T;V)}
\leq C(1+\varepsilon^{-1}+\rho^{1/2}).
\end{equation}
In view of equations~\eqref{eq:phase-convergence-l2}
and~\eqref{eq:phase-convergence-linfty}, we additionally have that
\begin{equation}
  \label{eq:another-convergence-varphi}
  \varphi_n\to\varphi_\varepsilon
  \quadd\text{ weakly* in }
  W^{1,\infty}(0,T;H)\cap H^1(0,T;V).
\end{equation}
This proves that $\varphi_\varepsilon$ belongs to
$W^{1,\infty}(0,T;H)\cap H^1(0,T;V)$.

In this second part we refine our argument, removing the dependence on
$\varepsilon$ in the estimate~\eqref{eq:further-regularity-epsilon}.
Like the former part of this proof, we want to derive
equation~\eqref{eq:phase-yosida}.  Since nothing ensures the
weak-derivability of $\Laplace\varphi_\varepsilon$, we take $v\in V$
and \pier{rewrite}~\eqref{eq:phase-yosida} using the Gauss theorem
\begin{equation*}
  (\partial_t\varphi_\varepsilon,v)
  +(\nabla\varphi_\varepsilon,\nabla v)
  +(\beta_\varepsilon(\varphi_\varepsilon),v)
  =(-\pi(\varphi_\varepsilon)
  +\gamma\partial_t w_\varepsilon,v).
\end{equation*}
At this point, since $\varphi_\varepsilon\in H^1(0,T;V)$ and the
considerations on the weak-de\-riv\-abil\-i\-ty of
$\beta_\varepsilon(\varphi_\varepsilon)$, $\pi(\varphi_\varepsilon)$
and $\partial_t w_\varepsilon$ remain valid,
$\partial_t\varphi_\varepsilon$ is derivable with respect to time.
We derive the above equation finding
\begin{equation*}
  (\partial_t^2\varphi_\varepsilon,v)
  +(\nabla\partial_t\varphi_\varepsilon,\nabla v)
  +(\beta_\varepsilon'(\varphi_\varepsilon)\partial_t\varphi_\varepsilon,v)
  =(\tilde g_4,v),
  \quadd\forall v\in V,
  \text{ a.e. in }(0,T),
\end{equation*}
where $\tilde g_4$ is defined in the same way as $g_4$ and it
satisfies $\norm{\tilde g_4}_{L^2(0,T;H)}\leq C(1+\rho^{1/2})$.  We
take $v=\partial_t\varphi_\varepsilon$ and after some calculations
carried out as in the former part we arrive at
\begin{equation*}
  \cfrac{1}{2}
  \norm{\partial_t\varphi_\varepsilon(t)}_H^2
  +\int_0^t
  \norm{\nabla\partial_t\varphi_\varepsilon}_H^2
  \leq
  \cfrac{1}{2}
  \norm{\partial_t\varphi_\varepsilon(0)}_H^2
  +C(1+\rho)
  +\int_0^t
  \norm{\partial_t\varphi_\varepsilon}_H^2.
\end{equation*}
Now, using the hypothesis $\beta^\circ(\varphi_0)\in H$ we deduce that
\begin{equation*}
  \norm{\beta_\varepsilon(\varphi_0)}_H
  \leq\norm{\beta^\circ(\varphi_0)}_H\leq C.
\end{equation*}
Hence, arguing as in
equation~\eqref{eq:further-regularity-initial-data}, we have that
$\norm{\partial_t\varphi_\varepsilon(0)}_H\leq C$ and \pier{use the Gronwall
lemma} to deduce~that
\begin{equation}
\norm{\varphi_\varepsilon}_{
  W^{1,\infty}(0,T;H)
  \cap
  H^1(0,T;V)}
\leq C(1+\rho^{1/2}).
\label{eq:partial-t-varphi-l-infty}
\end{equation}
By comparison, we find
$\norm{-\Laplace\varphi_\varepsilon+\xi_\varepsilon}_{L^\infty(0,T;H)}
\leq C(1+\rho^{1/2})$, \pier{whence}, by the \pier{same argument as
  in} the second a priori
estimate~\eqref{eq:second-a-priori-estimate}, we conclude that
$\Laplace\varphi_\varepsilon, \xi_\varepsilon\in L^\infty(0,T;H)$ and
that
\begin{equation}
\norm{\Laplace\varphi_\varepsilon}_{L^\infty(0,T;H)}
+\norm{\xi_\varepsilon}_{L^\infty(0,T;H)}
\leq C(1+\rho^{1/2}).
\label{eq:laplace-varphi-l-infty}
\end{equation}

\subsection{Further regularity for \problem{B}}%
Since the proof of Theorem~\ref{thrm:further-regularityB} is based on
the same idea of the previous proof, the present proof is just
sketched and we will focus only on the differences.
We consider equation~\eqref{eq:phaseB-faedo-galerkin} and, like
before, we want to derive it with respect to the time.
The terms $\Laplace\varphi_n$, $\beta_\varepsilon(\varphi_n)$,
$\pi(\varphi_n)$, and $\partial_tw$ are weakly-differentiable by the
considerations made in the previous subsection\marginnote{ ho detto
  subsection}.
If the term $\rho\Sign_\varepsilon(\varphi_n-\varphi^*_n)$ were
weakly-differentiable, we could carry out the proof as before; the
weak-derivability of this term is shown in the following lemma.%
\marginnote{Questo lemma serve per differenziare il segno.  Cerchiamo
  di evitare degli Stampacchia infinito-dimensionale.}
\begin{lem}
\label{lem:sign-derivative}
Let $\varepsilon>0$ and $u\in H^1(0,T;H)$.
Let $A=\{t\in [0,T]: \norm{\varphi(t)}_H\leq \varepsilon\}$ and
$B=[0,T]\backslash A$.
Then $\Sign_\varepsilon(u)\in H^1(0,T;H)$ and
\begin{align}
&
  \frac{\de}{\de t}\Sign_\varepsilon(u)=
  \begin{dcases}
    \frac{u_t}{\varepsilon}
  &\text{a.e.\ in }A,
  \\
  \frac{u_t}{\norm{u}_H}
  -\frac{(u,u_t)}{\norm{u}_H^3}u
  &\text{a.e.\ in }B.
  \end{dcases}
\end{align}
Moreover, \pier{we have that}
\begin{equation}
  \label{eq:sign-derivative-positive}
  \left(\frac{\de}{\de t}\Sign_\varepsilon(u),u_t\right)\geq0
  \quadd\text{ a.e.\ in }[0,T].
\end{equation}
\end{lem}
\begin{proof}
Let $l(t)=1/\max\{\varepsilon,\norm{u(t)}_H\}$.
By \pier{the} Stampacchia theorem, this function belongs to $H^1(0,T)$ and its
derivative is
\begin{align}
  \frac{\de l}{\de t}=
  \begin{dcases}
  0&\text{a.e.\ in }A,
  \\
  -\frac{(u,u_t)}{\norm{u}_H^3}
  \quadd&\text{a.e.\ in }B.
  \end{dcases}
\end{align}
By equation~\eqref{eq:sign-def}, we have that
$\Sign_\varepsilon(u(t))=l(t)u(t)$.
Hence, $\Sign_\varepsilon(u)$ is differentiable by Leibniz rule and
its derivative is
\begin{equation}
  \frac{\de}{\de t}\Sign_\varepsilon(u)
  =l_t u+ l u_t
  =
  \begin{dcases}
    \frac{u_t}{\varepsilon}&\text{a.e.\ in } A,
    \\
    -\frac{(u_t,u)}{\norm{u}_H^3}u+\frac{u_t}{\norm{u}_H}
    \quadd&\text{a.e.\ in } B.
  \end{dcases}
\end{equation}
Inequality~\eqref{eq:sign-derivative-positive} is trivial for a.a.\ $t\in A$;
on $B$ we use \pier{the} Cauchy--Schwartz inequality obtaining
\begin{equation*}
  \left(\frac{\de}{\de t}\Sign_\varepsilon(u),u_t\right)
  =
  -\frac{(u_t,u)^2}{\norm{u}_H^3}+\frac{\norm{u_t}_H^2}{\norm{u}_H}
  \geq
  -\frac{\norm{u_t}_H^2\norm{u}_H^2}{\norm{u}_H^3}
  +\frac{\norm{u_t}_H^2}{\norm{u}_H}=0.
  \qedhere
\end{equation*}
\end{proof}

At this point, taking \pier{Lemma~\ref{lem:sign-derivative} into account}, we have that
$\varphi_t$ is weakly-dif\-fe\-ren\-tiable by comparison and we can
derive equation~\eqref{eq:phaseB-faedo-galerkin} obtaining
\begin{equation}\label{eq:further-regularity-faedo-galerkinB}
  \begin{aligned}
  &\left(\partial_t^2\varphi_n-\Laplace\partial_t\varphi_n
  +\beta_\varepsilon'(\varphi_n)\partial_t\varphi_n
  +\rho\frac{\de}{\de t}\Sign_\varepsilon(\varphi_n-\varphi^*_n),v\right)=(g_4,v),
  \\&
  \quadd\forall v\in V_n,\text{ a.e. in }(0,T),
  \end{aligned}
\end{equation}
where
\begin{equation*}
  g_4:=-\pi'(\varphi_\varepsilon)\partial_t\varphi_n+
  \gamma\partial_t^2w_n
  ,\quad\text{ with }
  \norm{g_4}_{L^2(0,T;H)}\leq C(1+\varepsilon^{-1}+\rho^{1/2})
  .
\end{equation*}
We take $v=\partial_t \varphi_\varepsilon$ in
equation~\eqref{eq:further-regularity-faedo-galerkinB} and \pier{integrate} over $(0,t)$ obtaining
\begin{equation*}
  \begin{split}
    &\cfrac{1}{2}
    \norm{\partial_t\varphi_n(t)}_H^2
    +\int_0^t
    \norm{\nabla\partial_t\varphi_n}_H^2
    +\int_{Q_t}
    \beta_\varepsilon'(\varphi_n)
    |\partial_t\varphi_n|^2
    \\&
    +\rho\int_0^t((\Sign_\varepsilon(\varphi_n-\varphi_n^*))_t,
    \partial_t\varphi_n)
    \\&
    =
    \cfrac{1}{2}
    \norm{\partial_t\varphi_n(0)}_H^2
    +\int_0^t
    (g_4,\partial_t\varphi_n)
    \\&
    \leq
    \cfrac{1}{2}
    \norm{\partial_t\varphi_n(0)}_H^2
    +\cfrac{1}{2}
    \int_0^t
    \norm{\partial_t\varphi_n}_H^2
    +C(1+\varepsilon^{-2}+\rho).
  \end{split}
\end{equation*}
The inequality above differs
from~\eqref{eq:further-regularity-big-estimate} only for the presence
of the term involving the derivative of $\Sign_\varepsilon$.
This term is non-negative due to
inequality~\eqref{eq:sign-derivative-positive}.
We have to control the term $\norm{\partial_t\varphi_n(0)}_H$, which
can be treated similarly
as~\eqref{eq:further-regularity-initial-data}, with the only
difference that now we have the additional contribution
\begin{equation}
  \label{eq:further-regularity-rho}
  \rho\norm{\Sign_\varepsilon(\varphi_n-\varphi_n^*)}_H\leq \rho
\end{equation}
on the right-hand side.  Then, using~\eqref{eq:further-regularity-rho}
and applying the Gronwall lemma, we easily arrive at
\begin{equation}
  \label{eq:further-regularity-epsilonB}
  \norm{\varphi_n}_{
    W^{1,\infty}(0,T;H)
    \cap
    H^1(0,T;V)}
  \leq C(1+\varepsilon^{-1}+\rho).
\end{equation}
Possibly taking subsequences, like
in~\eqref{eq:another-convergence-varphi} we have that
\begin{equation}
  \varphi_n\to\varphi_\varepsilon
  \quadd\text{ weakly* in }
  W^{1,\infty}(0,T;H)\cap H^1(0,T;V),
\end{equation}
thus $\varphi_\varepsilon\in W^{1,\infty}(0,T;H)\cap H^1(0,T;V)$.

We \pier{rewrite} equation~\eqref{eq:phaseB-yosida} as
\begin{equation*}
  \begin{aligned}
    &
    (\partial_t\varphi_\varepsilon,v)
    +(\nabla\varphi_\varepsilon,\nabla v)
    +(\beta_\varepsilon(\varphi_\varepsilon),v)
    +\rho(\Sign_\varepsilon(\varphi_\varepsilon-\varphi^*),v)
    \\&
    =(-\pi(\varphi_\varepsilon)
    +\gamma\partial_t w_\varepsilon,v),
  \end{aligned}
\end{equation*}
\pier{derive} it, and test with $v=\partial_t\varphi_\varepsilon$.
After few calculations in the same style as before, the
estimate~\eqref{eq:further-regularity-epsilonB} improves to
\begin{equation}
\norm{\varphi_\varepsilon}_{
  W^{1,\infty}(0,T;H)
  \cap
  H^1(0,T;V)}
\leq C(1+\rho).
\end{equation}
By comparison, we find
$\norm{-\Laplace\varphi_\varepsilon+\xi_\varepsilon}_{L^\infty(0,T;H)}
\leq C(1+\rho)$.
We multiply $-\Laplace\varphi_\varepsilon+\xi_\varepsilon$ by
$-\Laplace\varphi_\varepsilon$, obtaining
\begin{equation*}
  \begin{aligned}
    &
  \norm{\Laplace\varphi_\varepsilon}_H^2
  +\int_\Omega\beta_\varepsilon'|\nabla\varphi_\varepsilon|^2
  \leq
  \norm{\Laplace\varphi_\varepsilon}_H
  \norm{-\Laplace\varphi_\varepsilon+\xi_\varepsilon}_H
  \\&
  \leq
  \frac 1 2
  \norm{\Laplace\varphi_\varepsilon}_H^2
  +C(1+\rho^2)
  \quadd
  \text{ a.e.\ in }(0,T).
  \end{aligned}
\end{equation*}
Thus, \pier{we conclude that}
$\norm{\Laplace\varphi_\varepsilon}_{L^\infty(0,T;W)}\leq C(1+\rho)$
by elliptic regularity and
$\norm{\xi_\varepsilon}_{L^\infty(0,T;H)}\leq C(1+\rho)$ by
comparison.

\section{Continuous dependence of the solutions}
\label{continuous-dependence}

We prove now Theorems~\ref{thrm:continuous-dependence}
and~\ref{thrm:continuous-dependenceB}.

\paragraph{Continuous dependence of the solutions for \problem{A}.}
In order to simplify the notation, we let
$\vartheta_0=\vartheta_{0,1}-\vartheta_{0,2}$ and analogously we
define\marginnote{ricontrollare} $w_0$, $\varphi_0$, $f$, $w$,
$\varphi$, $\xi$ and $\sigma$.  In this proof, $C$ denotes a
time-to-time-different, positive, large-enough constant independent of
the just-said data and of $\rho$.

It is clear that
\begin{align}
  &(w_t+l\varphi)_t
  -\kappa\Laplace w_t
  -\tau\Laplace w
  +\rho\sigma
  =f,
  \label{eq:heat-difference}
  \\
  &\varphi_t-\Laplace\varphi
  +\xi
  +\pi(\varphi_1)-\pi(\varphi_2)
  =\gamma w_t.
  \label{eq:phase-difference}
\end{align}
We multiply equations~\eqref{eq:heat-difference}
and~\eqref{eq:phase-difference} by $(w_t+l\varphi)$ and $\kappa
l^2\varphi$ respectively, sum up and integrate over $\Omega$ obtaining
\begin{equation*}
  \begin{split}
    &\cfrac{1}{2}
    \cfrac{\de}{\de t}
    \norm{w_t+l\varphi}_H^2
    +\kappa\norm{\nabla w_t}_H^2
    +\kappa l(\nabla w_t,\nabla\varphi)
    +\cfrac{\tau}{2}
    \cfrac{\de}{\de t}
    \norm{\nabla w}_H^2
    \\
    &+\tau l(\nabla w,\nabla\varphi)
    +\rho(\sigma,w_t+l\varphi)_H
    +\cfrac{\kappa l^2}{2}
    \cfrac{\de}{\de t}
    \norm{\varphi}_H^2
    \\&
    +\kappa l^2\norm{\nabla\varphi}_H^2
    +\kappa l^2 (\xi,\varphi)_H
    +\kappa l^2 (\pi(\varphi_1)-\pi(\varphi_2),\varphi)_H
    \\&
    =(f,w_t)+l(f,\varphi)
    +\gamma\kappa l^2
    (w_t,\varphi).
  \end{split}
\end{equation*}
We rearrange and \pier{use} the Lipschitz-continuity of $\pi$,
equations~\eqref{eq:sign} and~\eqref{eq:beta}, and the monotonicity of
$\Sign$ and $\beta$ to infer that
\begin{equation*}
  \begin{split}
    &\cfrac{1}{2}
    \cfrac{\de}{\de t}
    \norm{w_t+l\varphi}_H^2
    +\cfrac{\tau}{2}
    \cfrac{\de}{\de t}
    \norm{\nabla w}_H^2
    +\cfrac{\kappa l^2}{2}
    \cfrac{\de}{\de t}
    \norm{\varphi}_H^2
    \\&
    +\kappa(\norm{\nabla w_t}_H^2
    +l(\nabla w_t,\nabla\varphi)
    +l^2\norm{\nabla\varphi}_H^2)
    \\&
    \leq
    (f,w_t)+l(f,\varphi)
    +\gamma\kappa l^2
    (w_t,\varphi)
    +C\norm{\varphi}_H^2
    -\tau l(\nabla w,\nabla\varphi).
  \end{split}
\end{equation*}
At this point, \pier{the Young inequality and the fact that}
\begin{equation*}
l(\nabla w_t,\nabla\varphi)
\geq
-\cfrac{1}{2}
(\norm{\nabla w_t}^2_H+l^2\norm{\nabla\varphi}_H^2)
\end{equation*}
\pier{allow us} to deduce
\begin{equation*}
  \begin{split}
    &\cfrac{1}{2}
    \cfrac{\de}{\de t}
    \norm{w_t+l\varphi}_H^2
    +\cfrac{\kappa l^2}{2}
    \cfrac{\de}{\de t}
    \norm{\varphi}_H^2
    +\cfrac{\tau}{2}
    \cfrac{\de}{\de t}
    \norm{\nabla w}_H^2
    \\&
    +\frac{\kappa}{2}(\norm{\nabla w_t}_H^2
    +l^2\norm{\nabla\varphi}_H^2)
    \\&
    \leq
    C\norm{f(t)}_H^2
    +C(\norm{w_t}_H^2+\norm{\varphi}_H^2)
    +\frac{\kappa l^2}{4}\norm{\nabla\varphi}_H^2
    +\frac{\tau^2}{\kappa}\norm{\nabla w}_H^2.
  \end{split}
\end{equation*}
We integrate between $0$ and $t$
\begin{equation*}
  \begin{split}
    &\cfrac{1}{2}
    \norm{w_t(t)+l\varphi(t)}_H^2
    +\cfrac{\kappa l^2}{2}
    \norm{\varphi(t)}_H^2
    +\cfrac{\tau}{2}
    \norm{\nabla w(t)}_H^2
    \\&
    +\frac{\kappa}{2}\int_0^t\norm{\nabla w_t}_H^2
    +\cfrac{\kappa l^2}{4}\int_0^t \norm{\nabla\varphi}_H^2
    \\&
    \leq
    C\norm{f}_{L^2(0,T;H)}^2
    +\frac{1}{2}\norm{\vartheta_0+l\varphi_0}_H^2
    +\cfrac{\kappa l^2}{2}\norm{\varphi_0}_H^2
    +\cfrac{\tau}{2}\norm{\nabla w_0}_H^2
    \\&
    \quad +C\int_0^t(\norm{w_t}_H^2+\norm{\varphi}_H^2)
    +\frac{\tau^2}{\kappa}\int_0^t\norm{\nabla w}_H^2.
  \end{split}
\end{equation*}
Finally, we note that
\begin{align*}
  \frac \kappa{\kappa+2}\norm{w_t(t)}_H^2
  +\frac{\kappa l^2}{2}\norm{\varphi(t)}_H^2
  &\leq
  \norm{w_t(t)+l\varphi(t)}_H^2
  +\kappa l^2
  \norm{\varphi(t)}_H^2,
  \\
  \norm{\vartheta_0+l\varphi_0}_H^2
  +\kappa l^2
  \norm{\varphi_0}_H^2
  &\leq
  2\norm{\vartheta_0}_H^2+(\kappa+1)l^2\norm{\varphi_0}_H^2,
\end{align*}
and so we can apply the Gronwall Lemma finding
\begin{equation*}
  \begin{split}
  &\norm{w_t}_{L^\infty(0,T;H)}
  +\norm{\nabla w}_{L^\infty(0,T;H)}
  +\norm{\nabla w_t}_{L^2(0,T;H)}
  \\&
  +\norm{\varphi}_{L^\infty(0,T;H)}
  +\norm{\nabla\varphi}_{L^2(0,T;H)}
  \\&
  \leq
  C(\norm{f}_{L^2(0,T;H)}+
  \norm{\vartheta_0}_H+\norm{\varphi_0}_H
  +\norm{\nabla w_0}_H).
  \end{split}
\end{equation*}
\pier{As $w_0,\varphi_0\in V$, this implies} that
\begin{equation}
  \begin{split}
  &\norm{w}_{W^{1,\infty}(0,T;H)\cap H^1(0,T;V)}
  +\norm{\varphi}_{L^\infty(0,T;H)\cap L^2(0,T;V)}
  \\&
  \leq
  C(\norm{f}_{L^2(0,T;H)}
  +\norm{\vartheta_0}_H
  +\norm{w_0}_V
  +\norm{\varphi_0}_H
  ),
  \end{split}
\end{equation}
and the proof is complete.

We conclude this paragraph \pier{with the proof} of
Corollary~\ref{cor:uniqueness}.
Assume that $(w_i,\varphi_i,\xi_i,\sigma_i)$, $i=1, 2$, are two
solutions given by the existence Theorem~\ref{thrm:existence}.
Since $l=\alpha$ we can apply the just-proven
Theorem~\ref{thrm:continuous-dependence} with
$(\vartheta_{0,i},w_{0,i},\varphi_{0,i},f_i)=(\vartheta_{0},w_{0},\varphi_{0},f)$.
Hence, by the equation above we deduce that $w_1=w_2$ and
$\varphi_1=\varphi_2$.  By comparison, we conclude that
$\sigma_1=\sigma_2$ and $\xi_1=\xi_2$.

\paragraph{Continuous dependence of the solutions for \problem{B}.}
We do not give the proof of Theorem~\ref{thrm:continuous-dependenceB},
since \pier{it} goes along the same line of the previous one and no
further idea arises.
As before, continuous dependence implies uniqueness of the components
$w$ and $\varphi$.
If in addition $\beta$ is single-valued, $\xi$ is uniquely determined,
thus, by comparison in~\eqref{eq:phaseB}, the function
$\sigma$ is uniquely determined, as well.

\section{Existence of sliding modes}
This section is devoted to prove Theorems~\ref{thrm:sliding-mode}
and~\ref{thrm:sliding-modeB}.
In this \pier{section}, we deal with the approximated solutions
$(w_\varepsilon,\varphi_\varepsilon,\xi_\varepsilon,\sigma_\varepsilon)$
and we will take the limit as $\varepsilon\to 0$.
Before going through the proofs of the Theorems, we \pier{show} the following
Lemma.

\begin{lem}\label{lem:decreasing}
  Let $T, M>0$,\footnote{The first part of this lemma is still working
    if $M=0$.} $\psi_0\geq0$, and let $\psi:[0,T]\to\real$ be an a
  non-negative, absolutely continuous function with $\psi(0)=\psi_0$.
  Let $A$ be the set
  \begin{equation}
    A:=\{t\in[0,T]: \psi(t)>0\}.
  \end{equation}
  If $\psi'(t)\leq-M$ a.e.\ in $A$, then the following
  conclusions hold true.
  \begin{enumerate}
  \item
    If $\psi_0=0$, then $\psi\equiv0$.
  \item
    If $M>\psi_0/T$, then there exist a time $T^*\in(0,T)$ such that
    \begin{equation}
      \label{eq:lemma-on-T*}
      T^*\leq\frac{\psi_0}{M}<T,
    \end{equation}
    as well as the function $\psi$ is strictly decreasing in $[0,T^*)$
      and vanishes in $[T^*,T]$.
  \end{enumerate}
\end{lem}
\begin{proof}
  1)\quad Suppose on the contrary that $A$ is non empty. Let $B=(a,b)$
  be a connected component of $A$.  The function $\psi$ restricted to
  $B$ is strictly decreasing.  Indeed, if $a<t_0<t_1<b$, we have that
  \begin{equation*}
  \psi(t_1)-\psi(t_0)=\int_{t_0}^{t_1}\psi'(s)\de s\leq-M(t_1-t_0)<0.    
  \end{equation*}
  We now take the limit as $t_0\to a$ obtaining
  \begin{equation*}
    \psi(t_1)\leq\lim_{t_0\to a}\psi(t_0)=\psi(a)=0,
  \end{equation*}
  which is a contradiction for we assumed $\psi(t_1)>0$.

  2)\quad We may assume $\psi_0>0$, because the case $\psi_0=0$ follows
  directly from the former part with $T^*=0$.  We define $T^*$ as
  \begin{equation}
    \label{eq:sup-psi}
    T^*:=\sup\{t\in(0,T): \psi(s)>0\text{ for all }s\in(0,t)\}.
  \end{equation}
  By continuity of $\psi$, $T^*$ is well-defined and greater than $0$.
  Moreover, the interval $[0,T^*)$ is contained in $A$, hence
  we have that
  \begin{equation*}
    \psi(T^*)-\psi(0)=\int_0^{T^*}\psi'(t)\de t\leq-MT^*,
  \end{equation*}
  thus
  \begin{equation*}
    T^*\leq\frac{\psi_0-\psi(T^*)}{M}\leq\frac{\psi_0}{M}<T.
  \end{equation*}
  Note that $\psi$ is strictly decreasing in $[0,T^*)$ for what we
  have proven in 1).  It is clear that $\psi(T^*)=0$.  Indeed, if on
  the contrary $\psi(T^*)>0$, then $\psi>0$ in $[0,T^*+\varepsilon)$
  for a small $\varepsilon$ and the supremum in
  definition~\eqref{eq:sup-psi} fails.  Finally, we define
  $\delta:[0,T-T^*]\to[0,+\infty)$ as $\delta(t)=\psi(t+T^*)$ and we
  use the first part of the lemma, deducing $\delta=0$, thus
  $\psi(t)=0$ for all $t\in[T^*,T]$.
\end{proof}

\subsection{Existence of sliding modes for \problem{A}}

We define for $\varepsilon>0$, $\eta_\varepsilon,
g_{\varepsilon}:[0,T]\to H$ as
\begin{align*}
  \eta_\varepsilon&:=\partial_tw_\varepsilon
  +\alpha\varphi_\varepsilon-\eta^*,
  \\
  g_{\varepsilon}&:=\tau\Laplace w_\varepsilon
  -\kappa\alpha\Laplace\varphi_\varepsilon
  +(\alpha-l)\partial_t\varphi_\varepsilon-\kappa\Laplace\eta^*+f.
\end{align*}
Analogously, we define $\eta=\partial_tw+\alpha\varphi-\eta^*$.
Because of the
estimates~\eqref{eq:first-a-priori-estimate-without-epsilon},
\eqref{eq:laplace-varphi-l-infty},
\eqref{eq:partial-t-varphi-l-infty}, and the hypothesis on the target
function~\eqref{eq:target-function},\marginnote{forse migliorare} we
infer that $g_{\varepsilon}\in L^\infty(0,T;H)$ and
\begin{equation*}
\norm{g_{\varepsilon}}_{L^\infty(0,T;H)}\leq C_5(1+\rho^{1/2}),  
\end{equation*}
where the constant $C_5$ is defined in
equation~\eqref{eq:c5-definition}.
Recalling that
$\sigma_\varepsilon=\Sign_\varepsilon(\eta_\varepsilon)$ we rewrite
equation~\eqref{eq:heat-yosida} as
\begin{equation*}
  (\partial_t\eta_\varepsilon-\kappa\Laplace\eta_\varepsilon
  +\rho\sigma_\varepsilon,v)=(g_{\varepsilon},v)
  \quadd\forall v\in H,\text{ a.e.\ in }(0,T).
\end{equation*}
We take $v=\sigma_\varepsilon$ in the above equation and integrate
between $t$ and $t+h$ (with $h\in(0,T-t)$) obtaining
\begin{equation*}
    \begin{aligned}
      &
      \int_t^{t+h}(\partial_t\eta_\varepsilon,\sigma_\varepsilon)
      +\kappa\int_t^{t+h}(\nabla\eta_\varepsilon,\nabla\sigma_\varepsilon)
      +\rho\int_t^{t+h}\norm{\sigma_\varepsilon}_{H}^2
      =\int_t^{t+h}(g_{\varepsilon},\sigma_\varepsilon).
    \end{aligned}
\end{equation*}
Using the fact that $\Sign_\varepsilon$ is the Fr\'echet-differential
of $\norm{\cdot}_{H,\varepsilon}$, we \pier{deduce} that
\begin{equation}
  (\partial_t\eta_\varepsilon,\sigma_\varepsilon)
  =\frac{\de}{\de t}\norm{\eta_\varepsilon}_{H,\varepsilon}
  .\label{eq:norm-derivative}
\end{equation}
In view of~\eqref{eq:sign-def}, we have that
\begin{equation}
    \int_t^{t+h}(\nabla\eta_\varepsilon,\nabla\sigma_\varepsilon)=
    \int_t^{t+h}
    \frac{\norm{\nabla\eta_\varepsilon}_H^2}
         {\max\{\varepsilon,\norm{\eta_\varepsilon}_H\}}\geq 0.
    \label{eq:positive}
\end{equation}
Finally, as $\norm{\sigma_\varepsilon}_H\leq1$ a.e.\ in $(0,T)$, \pier{it follows} that
\begin{equation*}
    \int_t^{t+h}(g_{\varepsilon},\sigma_\varepsilon)\leq
    hC_5(1+\rho^{1/2}).
\end{equation*}
Putting everything together we obtain
\begin{equation}
  \label{eq:dav1}
  \norm{\eta_\varepsilon(t+h)}_{H,\varepsilon}
  -\norm{\eta_\varepsilon(t)}_{H,\varepsilon}
  +\rho\int_t^{t+h}\norm{\sigma_\varepsilon}_{H}^2
  \leq hC_5(1+\rho^{1/2}).
\end{equation}
We remark that $\eta_\varepsilon$ converges to $\eta$ in
$C^0([0,T];H)$ as $\varepsilon\to0$,
(cf.~\eqref{eq:w-convergence-c1}--\eqref{eq:phase-convergence-c0}\pier{)},
which can be replicated for the limit as $\varepsilon\to0$).
Moreover, the Yosida approximations $\norm{\cdot}_{H,\varepsilon}$
converge to $\norm{\cdot}_H$ uniformly in $H$,
by~\eqref{eq:quasi-norm-uniform}.
It follows that the first two terms of the previous inequality converge.
At this point, we take the inferior limit in~\eqref{eq:dav1} as
$\varepsilon\to 0$ and we use \pier{the} Lebesgue dominate convergence theorem,
\marginnote{non mi ricordo mai qual \`e quello
  della convergenza dominata e quale della convergenza monotona}
property~\eqref{eq:sigma-convergence}, and the weak lower
semicontinuity of norms obtaining
\begin{equation*}
  \norm{\chi(t+h)}_H
  -\norm{\chi(t)}_H
  +\rho\int_t^{t+h}\norm{\sigma}_{H}^2
  \leq hC_5(1+\rho^{1/2}).
\end{equation*}
We divide by $h$ and we take the limit as $h\to 0$
\begin{equation*}
    \frac{\de}{\de t}(\norm{\eta(t)}_H)+\rho\norm{\sigma(t)}_H^2
    \leq C_5(1+\rho^{1/2}) \quadd\text{ for a.a.\ }t\in(0,T).
\end{equation*}
We introduce the function $\psi(t)=\norm{\eta(t)}_H$ and the quantity
\begin{equation*}
  M(\rho)=\rho-C_5-C_5\rho^{1/2}.
\end{equation*}
We also set (see~\eqref{eq:psi0-def})
$\psi_0=\norm{\vartheta_0+\alpha-\eta^*}_H$.  The inequality above
implies
\begin{equation}
    \psi'(t)\leq-M(\rho),\quadd\text{ for a.a.\ $t$ in }\{t: \psi(t)>0\}.
\end{equation}
Using the Young inequality we obtain
\begin{equation}
  \label{eq:bound-on-M}
  M(\rho)\geq\frac{\rho}{2}-C_5-\frac{C_5^2}{2}.
\end{equation}
Thus, if we \pier{choose}
\begin{equation*}
\rho^*=2\left(\frac{\psi_0}{T}+C_5+\frac{C_5^2}{2}\right),
\end{equation*}
then for every $\rho>\rho^*$ \pier{it turns out that}  $M(\rho)>\psi_0/T$.

Finally we can use Lemma~\ref{lem:decreasing}, that guarantees the
existence of $T^*<T$ such that $\psi$ vanishes in $[T^*,T]$, i.e.\ the
thesis.  Moreover, the second part of the Lemma and
equation~\eqref{eq:bound-on-M} lead to
\begin{equation*}
  T^*\leq\frac{2\psi_0}{\rho-2C_5-C_5^2}<T.
\end{equation*}

\subsection{Existence of sliding modes for \problem{B}}

We define for $\varepsilon>0$, $\chi_\varepsilon,
g_{\varepsilon}:[0,T]\to H$ as
\begin{align*}
  \chi_\varepsilon&:=\varphi_\varepsilon-\varphi^*,
  \\
  g_{\varepsilon}&:=\gamma\partial_tw_\varepsilon
  -\beta_\varepsilon(\varphi^*)-\pi(\varphi_\varepsilon)+\Laplace\varphi^*.
\end{align*}
Analogously, we define $\chi=\varphi-\varphi^*$.
Like in the previous proof, we have that $g_{\varepsilon}\in
L^\infty(0,T;H)$ and
\begin{equation*}
\norm{g_{\varepsilon}}_{L^\infty(0,T;H)}\leq C_{10},
\end{equation*}
where the constant is \pier{specified in~\eqref{eq:c9-definition}}.
Recalling 
the definition of $\sigma_\varepsilon$,
we rewrite
equation~\eqref{eq:phaseB-yosida}~as
\begin{equation*}
  (\partial_t\chi_\varepsilon-\Laplace\chi_\varepsilon
  +\beta_\varepsilon(\varphi_\varepsilon)-\beta_\varepsilon(\varphi^*)
  +\rho\sigma_\varepsilon,v)=(g_{\varepsilon},v)
  \quadd\forall v\in H,\text{ a.e.\ in }(0,T).
\end{equation*}
We take $v=\sigma_\varepsilon$ in the equation above and integrate
between $t$ and $t+h$ (with $h\in(0,T-t)$) obtaining
\begin{equation*}
    \begin{aligned}
      &
      \int_t^{t+h}(\partial_t\chi_\varepsilon,\sigma_\varepsilon)
      +\kappa\int_t^{t+h}(\nabla\chi_\varepsilon,\nabla\sigma_\varepsilon)
      +\int_t^{t+h}(\beta_\varepsilon(\varphi_\varepsilon)-\beta_\varepsilon(\varphi^*),\sigma_\varepsilon)
      \\&
      +\rho\int_t^{t+h}\norm{\sigma_\varepsilon}_{H}^2
      =\int_t^{t+h}(g_{\varepsilon},\sigma_\varepsilon).
    \end{aligned}
\end{equation*}
Inequalities~\eqref{eq:norm-derivative} and~\eqref{eq:positive} are
still valid if we substitute $\eta_\varepsilon$ with
$\chi_\varepsilon$.
Moreover, using the monotonicity of $\beta_\varepsilon$, we have that
\begin{equation*}
  \int_t^{t+h}(\beta_\varepsilon(\varphi_\varepsilon)
  -\beta_\varepsilon(\varphi^*),
  \sigma_\varepsilon)=
  \int_t^{t+h}
  \frac{(\beta_\varepsilon(\varphi_\varepsilon)
    -\beta_\varepsilon(\varphi^*),\chi_\varepsilon)}
  {\max\{\varepsilon,\norm{\chi_\varepsilon}_H\}}\geq
  0.
\end{equation*}
Finally, as $\norm{\sigma_\varepsilon}_H\leq1$ a.e.\ in $(0,T)$, we
have that
\begin{equation*}
    \int_t^{t+h}(g_{\varepsilon},\sigma_\varepsilon)\leq h C_{10}.
\end{equation*}
Putting everything together we obtain
\marginnote{qui la dimostrazione \`e diversa dal caso
  precedente. penso di cambiarlo anche per il caso A}
\begin{equation}
  \label{eq:dav2}
  \norm{\chi_\varepsilon(t+h)}_{H,\varepsilon}
  -\norm{\chi_\varepsilon(t)}_{H,\varepsilon}
  +\rho\int_t^{t+h}\norm{\sigma_\varepsilon}_{H}^2
  \leq hC_{10}.
\end{equation}
For the same considerations made in the previous proof, we may take
the limit as $\varepsilon\to0$, obtaining
\begin{equation*}
  \norm{\chi(t+h)}_H
  -\norm{\chi(t)}_H
  +\rho\int_t^{t+h}\norm{\sigma}_{H}^2
  \leq hC_{10}.
\end{equation*}
We divide by $h$ and, using the Lebesgue differentiation theorem, we
take the limit as $h\to 0$
\begin{equation*}
    \frac{\de}{\de t}\norm{\chi(t)}_H+\rho\norm{\sigma(t)}_H^2
    \leq C_{10} \quadd\text{ for a.a.\ }t\in(0,T).
\end{equation*}
We introduce the function $\psi(t)=\norm{\chi(t)}_H$ and the quantity
\begin{equation*}
  M(\rho)=\rho-C_{10}.
\end{equation*}
We also set $\psi_0=\norm{\varphi_0-\varphi^*}_H$
(see~\eqref{eq:psi0-definitionB}).
Clearly, \pier{it holds that}
\begin{equation}
  \psi'(t)\leq-M(\rho),\quadd\text{ for a.a.\ $t$ in }\{t: \psi(t)>0\},
\end{equation}
so we can \pier{choose}
\begin{equation*}
\rho^*=\frac{\psi_0}{T}+C_{10},
\end{equation*}
\pier{and consequently} $M(\rho)>\psi_0/T$, for every $\rho>\rho^*$.

Finally we can use Lemma~\ref{lem:decreasing}, that guarantees the
existence of $T^*<T$ such that $\psi$ vanishes in $[T^*,T]$, i.e.\ the
thesis.  Moreover, the second part of the Lemma leads to
\begin{equation*}
  T^*\leq\frac{\psi_0}{\rho-C_{10}}<T.
\end{equation*}

\pcol{\subsection*{Acknowledgments}
The current contribution originated from the work done by
Davide Manini for the preparation of his master thesis, which 
has been discussed at the University of Pavia on July 2019.
Actually, the paper turns out to offer some extension to the 
results there contained.
The research of Pierluigi Colli is supported by the Italian Ministry of Education,
University and Research~(MIUR): Dipartimenti di Eccellenza Program (2018--2022)
-- Dept.~of Mathematics ``F.~Casorati'', University of Pavia.
In addition, PC gratefully acknowledges some other
support from the GNAMPA (Gruppo Nazionale per l'Analisi Matematica,
la Probabilit\`a e le loro Applicazioni) of INdAM (Istituto
Nazionale di Alta Matematica) and the IMATI -- C.N.R. Pavia, Italy.}

\end{document}